\documentclass[12pt,leqno]{article}
	\usepackage{amsmath}
	\usepackage{amssymb}

	\usepackage{theorem}

	\usepackage{url}
        \usepackage{esint}
        \newtheorem{thm}{Theorem}[section]
        \newtheorem{lma}{Lemma}[section]
	
        \newtheorem{prop}{Proposition}[section]
        \newtheorem{cor}{Corollary}[section]
\newenvironment{proof}{\medskip\noindent{\bf Proof.}}{\medskip}
	
	\theoremstyle{definition}
        \newtheorem{dfn}{Definition}[section]

	\newtheorem{rmk}{Remark}[section]

        \newcommand{\ms}{\medskip}
\numberwithin{equation}{section}

	\newcommand{\R}{{\mathbb R}}
	\newcommand{\N}{{\mathbb N}}
	
	\renewcommand{\o}{{\omega}}

	\renewcommand{\O}{{\mathcal O}}
	\newcommand{\D}{{\mathcal D}}
	\newcommand{\U}{{\mathcal U}}

	\renewcommand{\ss}{\subset}
	\newcommand{\sse}{\subseteq}
	\newcommand{\bs}{\backslash}
	\newcommand{\wt}{\widetilde}
	\newcommand{\wh}{\widehat}
	\newcommand{\ol}{\overline}
	\newcommand{\s}{S_i^{(l)}}
        
\newcommand{\cU}{\mathcal U}
\newcommand{\cO}{\mathcal O}
\newcommand{\om}{\omega}
	\newcommand{\e}[1]{\emph{#1}}

	\renewcommand{\div}{\,\mathrm{div}\,}

\begin{document}
\title{Square Functions and the $A_\infty$ Property of Elliptic Measures}
\author{C. Kenig\footnote{C.K. was partially supported by NSF grants DMS-0968472 and DMS-1265249.}
, B. Kirchheim, J. Pipher \& T. Toro\footnote{T.T. was partially supported by an NSF grant DMS-0856687, the Simons Foundation grant  228118 and the Robert R. \& Elaine F. Phelps Professorship in Mathematics.}}
\maketitle

\renewcommand{\thepage}{\arabic{page}}

\section{Introduction}

In this paper, we provide a new means of establishing solvability of the Dirichlet problem on Lipschitz domains, with measurable data, for second order elliptic, non-symmetric divergence form operators. 
We will show that a certain optimal Carleson measure estimate for bounded solutions of such operators implies a regularity result for the associated elliptic measure. 

\ms

We consider divergence form elliptic operators $L$ in $\R^n_+$, where $L=\div A(X)\nabla$, and $A(X)=\left( a_{ij}(X) \right)_{i,j=1}^n$ is a real $n\times n$ matrix with $a_{ij}\in L^\infty$ satisfying the uniform ellipticity condition:  there exists $ \lambda>0$  such that  for all $\xi\in \R^n$, one has
    \begin{equation}\label{ellip}
        \lambda\lvert \xi\rvert^2 \leq \langle A(X)\xi, \xi\rangle \leq \lambda^{-1}\lvert \xi\rvert^2.
    \end{equation}

The matrix $A$ is not be assumed to be symmetric. For future reference, when we say that a bound depends only on {\emph{the ellipticity}}, we will mean that it depends only on $\lambda$ and $\sup_{i,j}\lVert a_{ij}\rVert_{L^\infty}$.

\ms

We begin by briefly recalling of some of the properties of solutions to such operators, and refer the reader to previous literature for the details. In particular, we will use the definitions and results in Section 1 of \cite{KKPT}.
Since the coefficients of $L$ are merely bounded and measurable, solutions to $Lu = 0$ are initially defined in a weak sense. However, by the fundamental work of De Giorgi, Nash and Moser, weak solutions are Holder continuous in the interior of some order that depends only on the ellipticity of the operator, and positive solutions satisfy a Harnack principle. 
The results of Littman, Stampacchia, and Weinberger 
\cite{LSW} are also
valid in the non-symmetric setting. In particular, a Lipschitz domain $\Omega$ is
regular for the Dirichlet problem. That is, let $u_g$ denote the weak solution of $Lu_g=0$ in $\Omega$ with $u=g$, for $g$ continuous on $\partial\Omega$. Then the map $g \rightarrow u_g(X)$ is a positive bounded linear functional which in turn is represented by a probability measure $\o^X$.

Returning for the moment to the setting of the upper half space, we see that
for any bounded Borel measurable function $F$ on $\R^{n-1}$, one can uniquely solve the Dirichlet problem (by integration against the elliptic measure). That is if
\begin{equation*}
        \tag{DP}
        \label{eqn.dp}
        \left\{
        \begin{aligned}
            Lu&= 0 \text{  in  }\R_+^n\\
            u\vert_{\partial \R_+^n}&= F\text{  on  }\partial\R_+^n
        \end{aligned}
        \right.
\end{equation*}
then for $X\in\R^n_+$
\begin{equation}\label{sol-dp}
u(X)=\int_{\partial\R_+^n} F\, d\o^X,
\end{equation}
where $\o^X$ is the elliptic measure with pole at $X$.

\ms

By Harnack's principle, this family of measures is mutually absolutely continuous. 
We will be concerned with the further regularity properites of the measures $\o^X$, such as mutual absolute continuity with respect to Lebesgue measure on $\partial\R_+^n$.  The solvability of the Dirichlet problem for $L$ with data in $L^p(dx)$ is
characterized by means of a precise relationship between these elliptic measures associated to $L$ and Lebesgue measure.  These relationships quantify absolute continuity and are expressed in terms of the Muckenhoupt weight classes, $A_{\infty}$ or $A_p$. In this paper, the focus is on a new characterization of the property that $\o^X $ belongs to $A_{\infty}(dx)$ in terms of a Carleson measure property of bounded solutions. 

\ms

In order to describe this more concretely, let us recall some definitions.

\begin{dfn}\label{dfn.cm}
A measure $\o$ defined on $\R^{n-1}$ belongs to the weight class $A_{\infty}(dx)$ if any of the following equivalent conditions hold:
   \begin{enumerate}
        \item [(i)]  For every $\varepsilon\in (0,1)$ there exists $\delta\in(0,1)$ such that for any 
      cube $Q\subset \R^{n-1}$ and $E\subset Q$ with 
     \begin{equation} \label{eqn.1.1}
\frac{\o(E)}{\o(Q)}<\delta   \hbox{    then   }  \frac{|E|}{|Q|}<\varepsilon.
          \end{equation}
        \item [(ii)] For every $\varepsilon\in (0,1)$ there exists $\delta\in(0,1)$ such that for any 
      cube $Q\subset \R^{n-1}$ and $E\subset Q$ with 
     \begin{equation} \label{eqn.2.1}
\frac{|E|}{|Q|}<\delta   \hbox{    then   }  \frac{\o(E)}{\o(Q)}<\varepsilon.
          \end{equation}
        \item [(iii)] there exists a $p > 1$ such that $\o$ belongs to $A_p(dx)$.
    \end{enumerate}
\end{dfn}

\begin{dfn} 
The measure $\mu$ is a Carleson measure in the upper half space $\R_+^n$ if there exists a 
constant $C$ such for all cubes $Q \subset \partial\R_+^n$, 
$\mu(T(Q)) < C|Q|$, where $|Q|$ denotes the Lebesgue measure of the cube $Q$, and
$T(Q)=\left\{ X=(x,t), x\in S, 0\leq t\leq \ell(Q) \right\}$.
\end{dfn}

In \cite{F}, C. Fefferman discovered a property of harmonic functions, which has proven to be a powerful tool in analysis and potential theory. Namely, if $u(x,t)$ is the Poisson extension of $f \in BMO$, then 
$d\mu = t|\nabla u|^2 dxdt$ is a Carleson measure in the upper half space $\R_+^n$. 
The converse also holds for functions that are not too large at $\infty$. 
This fact has been generalized in a variety of ways: for harmonic functions on Lipschitz domains (\cite{FN}), and for more general second order elliptic operators whose elliptic measure has some regularity with respect with to the boundary Lebesgue measure (\cite{DKP}). 
The main result in \cite{DKP} established the equivalence between such a Carleson measure property of a solution to $Lu = 0$ with boundary data in $BMO$ and the $A_{\infty}$ property of the elliptic measure. In this paper, we show that this equivalence remains true when the data is merely assumed to be bounded and the Carleson measure is estimated by the $L^{\infty}$ norm instead of the (smaller) $BMO$ norm of the data. See
the statement in Corollary \ref{cor.1.13}.

\ms

The strategy of the present paper is modeled on, and extends, a geometric construction developed in \cite{KKPT}, where boundary value problems for
 non-symmetric, real divergence form operators were first systematically studied. In that paper, a closely related Carleson property of bounded solutions was shown to be equivalent to the $A_{\infty}$ property of elliptic measure. Roughly speaking, that property involved an $L^1$ version of the Carleson condition for gradients of approximants to bounded solutions. Such a  Carleson condition on approximants had appeared earlier in the literature in connection with $H^1 - BMO$ duality and the Corona Theorem. (\cite{G}, Chapter VIII.)  The key to proving (\ref{eqn.1.1}) in Definition \ref{dfn.cm} lay in constructing a function whose oscillation was large on a set of small $\o$-measure, and deriving a pointwise $L^1$ lower bound on the gradient of the solution with this oscillating data. (By contrast, the strategy in \cite{DKP} was to prove $A_{\infty}$ through  (\ref{eqn.2.1}) by constructing a function with small $BMO$ norm on a set of small Lebesgue measure.) A main contribution of the present paper is the construction, in Section 2,  of an oscillating data function on a set of small elliptic measure for which the C. Fefferman  $L^2$-type Carleson condition on gradients of solutions is large.

\ms

The following question for dyadic martingales in $[0,1]$ provided us with a model problem: given a set $E\subset [0,1]$
of measure zero, find a bounded dyadic martingale which is infinite on $E$. We provide a positive (quantitative) answer to this 
question in Section \ref{sec.0}, which then led us to the desired construction.

\ms

The Carleson measure conditions we consider here are essentially localized integrals of square functions. The square function, defined in Section 3 below, has played a substantial role in solvability of boundary value problems since its appearance in classical complex function theory, where it is referred to as an ``area integral".  In the classical setting of harmonic functions in Euclidean spaces, the square function and the non-tangential maximal function have equivalent $L^p$ norms for any $p > 0$ (\cite{S}). 
In the more general setting of solutions to second order elliptic operators $L$, the comparability of $L^p$ norm bounds (on all subdomains) is in fact equivalent to the $A_{\infty}$ property of the elliptic measure associated to $L$ (\cite{KKPT}). As a result of the construction in this paper, this last statement can be refined: $A_{\infty}$ follows from a one sided norm estimate of the square function by the non-tangential maximal function on the boundary (Theorem \ref{thm.1.16}).

\ms

Square function estimates in the upper half space are dealt with in Section 3, following the geometric constructions leading to the data function with a large oscillation. Section 4 contains the generalizations to Lipschitz domains.

\section{Functions with large oscillations on small sets}\label{sec.0}

\vskip .2in
\begin{lma}\label{lma.0.1}
    Let $\o$ be a doubling measure supported in all of $\R^n$, then $\o(\partial Q)=0$ for any dyadic cube $Q$.
\end{lma}

Let us introduce some notation. Let $Q_0=[0,1]^n\ss \R^n$ and $\text{int} \,Q_0=\O_0$. Let $\D_0$ the collection of dyadic cubes in $Q_0$.

\begin{proof}
    It is enough to do the proof in the case of the unit cube $Q_0$. Note that $\o(\O_0\backslash[\eta,1-\eta]^n)$ goes to $0$ as $\eta$ goes to $0$ and that for each $x\in \partial Q_0$ there is a $y$ such that $\lvert x-y\rvert < \frac{\sqrt{n}\eta}{2}$ and $B(y,\frac{\eta}{3})\ss \O_o\backslash[\eta,1-\eta]^n$. So take any maximal system $B_j, j\in J$ of disjoint balls of radius $\frac{\eta}{3}$ in $\O_o\backslash[\eta,1-\eta]^n$. It is now clear that the enlarged balls $(2+3\sqrt{n})B_j, j\in J$ cover all of $\partial Q_0$ and thus $\o(\partial Q_0) \leq \sum_{j\in J}\o( (2+3\sqrt{n})B_j)\leq C\sum_{j\in J}\o(B_j) \leq \o (\O_0\backslash[\eta,1-\eta]^n)$ which tends to $0$ as $\eta\to 0$.
\end{proof}

\begin{dfn}\label{dfn.0.1}
    Let $\varepsilon_0>0$ be given and small. If $E\ss \O_0$, a good $\varepsilon_0$-cover for $E$ of length $k$ is a collection of nested open sets $\left\{ \O_i \right\}_{i=1}^k$ with $E\sse \O_k \sse \O_{k-1} \sse \cdots \sse \O_1\sse \O_0$ such that for $l=1,\dots,k$,
    \begin{enumerate}
        \item [(i)] $\O_l\sse \bigcup_{i=1}^\infty S_i^{(l)}, \bigcup_{i=1}^\infty S_i^{(l)}\backslash \O_l \ss \partial Q_0$, where each $S_i^{(l)}$ is a dyadic cube in $\R^n$,
        \item [(ii)] $\bigcup_{i=1}^\infty S_i^{(l)} \ss \bigcup_{i=1}^\infty S_i^{(l-1)}$ and
        \item [(iii)] for all $1\leq l\leq k$, $\o (\O_l\bigcap S_i^{(l-1)})\leq \varepsilon_0\o(S_i^{(l-1)}).$
    \end{enumerate}
\end{dfn}

In the definition above the only difference between $\O_l$ and $\bigcup_{i=1}^\infty S_i^{(l)}$ is that we require $\O_l$ to be an open subset of the unit cube while the dyadic cubes $S_i^{(l)}$ are assumed to be closed and might intersect the boundary of the unit cube, condition (i) reflects this fact. 
Condition (ii) ensures that for each $\s$ there exists $S_j^{(l-1)}$ such that $\s\subset S_j^{(l-1)}$.
Note that condition (iii) in Definition \ref{dfn.0.1} above implies that each $S_i^{(l)}$ is properly contained in some $S_j^{(l-1)}$. To see this, observe that since $\O_l\ss \O_{l-1}$, $S_i^{(l)}$ must intersect some $S_j^{(l-1)}$. The inclusion $S_j^{(l-1)}\ss S_i^{(l)}$ is not possible for $\o (S_j^{(l-1)})\leq \o (S_j^{(l-1)}\cap \O_l)$ and (iii) gives a contradiction. If in Definition \ref{dfn.0.1} above we can take $k=+\infty$ then $\left\{ \O_l \right\}$ is called a good cover of infinite length.

\begin{lma}\label{lma.0.2}
    If $\left\{ \O_l \right\}$ is a good $\varepsilon_0$-cover of $E$ of length $k$ and $k\geq l>m\geq 1$, then $\o(S_j^{(m)}\cap \O_l)\leq \varepsilon_0^{l-m}\o(S_j^{(m)})$.
\end{lma}
\begin{proof}
    From the remark following the definition above, we have
    \begin{equation*}
        \O_{m+1}\cap S_j^{(m)} \sse \bigcup \left\{ S_i^{(m+1)}: S_i^{(m+1)}\ss S_j^{(m)} \right\}
    \end{equation*}
    and the inequality (iii) in Definition \ref{dfn.0.1} can be iterated $l-m$ times.
\end{proof}

\begin{lma}\label{lma.0.3}
    Given $\varepsilon_0>0$, there exists $\delta_0>0$ such that if $E\sse \O_0$ and $\o(E)\leq \delta_0$, then $E$ has a good $\varepsilon_0$-cover of length $k$, with $k\to \infty$ as $\omega(E)\to 0$. (In fact, $k\sim \frac{\log \o(E)}{\log\varepsilon_0}.$) 
\end{lma}

\begin{proof}
    Let $0< \varepsilon_0'<1$ be fixed, to be determined later. Let $U\sse \O_0$ be an open set containing $E$ such that $\o(U)<2\o(E)$ if $\o(E)>0$ or as small as needed if $\o(E)=0$. Set
    \begin{equation} \label{eqn.0.1}
        \widehat\O_k = \left\{ x: M_\o(X_U)(x)>\varepsilon_0' \right\}, 
    \end{equation}
    where
    \begin{equation}
        M_\o(g)(x)=\sup \left\{ \frac{1}{\o(B)}\int_B g\, d\o: B\text{ ball containing }x, B\sse \O_0 \right\}.
        \label{eqn.0.2}
    \end{equation}
    Note that $\widehat\O_k$ is open and since $U$ is open then $U\sse \widehat O_k$. Moreover since $\o$ is doubling,
    \begin{equation}
        \o(\widehat\O_k) \leq \frac{C}{\varepsilon_0'}\o(U)< \frac{2C}{\varepsilon_0'}\o(E).
        \label{eqn.0.3}
    \end{equation}
    If
    \begin{equation}
        \frac{2C}{\varepsilon_0'}\o(E)\leq \frac{2C\delta_0}{\varepsilon_0'}<\frac{1}{8},
        \label{eqn.0.4}
    \end{equation}
    then $\widehat\O_k$ has a Whitney decomposition, $\widehat\O_k=\bigcup_{i}\widehat S_i^{(k)}$, and for each $\widehat S_i^{(k)}$ there exists a point $P_i^{(k)}\in \widehat\O_k^c$ such that $\text{dist }(P_i^{(k)}, \widehat S_i^{(k)}) \simeq \text{diam }(\widehat S_i^{(k)}).$ Since $P_i^{(k)}\in \widehat \O_k^c$, for any ball $B$ containing $P_i^{(k)}$, $\frac{\o(U\cap B)}{\o(B)}\leq \varepsilon_0'$. Thus choosing a ball of radius comparable to $\text{diam}\, \widehat S_i^{(k)}$ and containing $\widehat S_i^{(k)}$, we have
    \begin{equation}
        \frac{\o(U\cap \widehat S_i^{(k)})}{\o(\widehat S_i^{(k)})}\leq C\varepsilon_0' < \varepsilon_0,
        \label{eqn.0.5}
    \end{equation}
    taking $\varepsilon_0'$ small enough depending only on the doubling constant of $\o$. Thus, given $\varepsilon_0$, select $\varepsilon_0'$ so that \eqref{eqn.0.5} holds then select an $\delta_0$ so that \eqref{eqn.0.4} also holds. Let $k$ be the largest integer such that $\left( \frac{C}{\varepsilon_0'} \right)^k\o(E)<\frac{1}{8}$. Let
    \begin{equation}
        \O_k = \O_0\cap \widehat \O_k, \text{ and }S_i^{(k)}=\widehat S_i^{(k)} \text{ if } \widehat S_i^{(k)}\subset Q_0.
        \label{eqn.0.6}
    \end{equation}
Note that $U\ss \O_k$. Since $\o(\partial Q_0)=0$, $\O_l\sse \bigcup_{i=1}^\infty S_i^{(l)}$ with $\bigcup_{i=1}^\infty S_i^{(l)}\backslash \O_l \ss \partial Q_0$ and \eqref{eqn.0.5} still holds with $S_i^{(k)}$ in place of $\widehat S_i^{(k)}$. We conclude that
\begin{equation}
    \o(U)\leq \varepsilon_0\o(\O_k).
    \label{eqn.0.7}
\end{equation}
Moreover \eqref{eqn.0.3} ensures that
\begin{equation}
    \o(\O_k)\leq \frac{C}{\varepsilon_0'}\o(U)\leq \frac{2C}{\varepsilon_0'}\o(E).
    \label{eqn.0.8}
\end{equation}
For $k-1\leq j\leq 1$, with $k$ as above, set
\begin{equation}
    \widehat \O_{j-1}=\left\{ x: M_\o(X_{\O_j})>\varepsilon_0' \right\} \text{ and }\O_{j-1}=\O_0\cap \widehat \O_{j-1}.
    \label{eqn.0.9}
\end{equation}
Our choice of $\delta_0$ ensures that $\widehat \O_{j-1}$ admits a dyadic decomposition $\widehat\O_{j-1}=\cup_i \widehat S_i^{(j-1)}$. As before we select $\{ S_i^{(j-1)} \}_i$ which satisfy $\o(\O_j\cap S_i^{(j-1)})\leq \varepsilon_0\o(S_i^{(j-1)})$ (see \eqref{eqn.0.5}). Note that \eqref{eqn.0.8} yield
\begin{equation}
    \o(\O_{j-1}) \leq \frac{C}{\varepsilon_0'}\o(\O_j)\leq \left( \frac{C}{\varepsilon_0'} \right)^{k-j}\o(\O_k)\leq 2\left( \frac{C}{\varepsilon_0'} \right)^{k-j+1}\o(E).
    \label{eqn.0.10}
\end{equation}
It is straightforward to show that $\left\{ \O_i \right\}_{i=1}^k$ is a good $\varepsilon_0$-cover for $E$. Note that by the definition of $k$, $\left( \frac{C}{\varepsilon_0'} \right)^k\o(E)\sim 1$, that is by our choice of $\varepsilon_0'$ in terms of $\varepsilon_0$, $k\sim \frac{\log\o(E)}{\log \varepsilon_0}$ and $k\to \infty$ as $\o(E)\to 0$.

\end{proof}
We need to introduce some additional notation. In particular we describe a way to select half of the children of any given dyadic cube. Note that any dyadic cube in $\R^n$ is a translation and dilation of the unit cube $[0,1]^n$. In particular if $Q\ss Q_0$ of side length $\ell(Q)=2^{-m}$ for some $m\geq 0$ then $Q=\Pi_{i=1}^n\left[ \frac{\tau_i}{2^{-m}}, \frac{\tau_i+1}{2^{-m}} \right]$ with $0\leq \tau_i<2^{-m}-1$. Let $e_Q=\left( \frac{\tau_1}{2^{-m}},\dots,\frac{\tau_n}{2^{-m}}\right)$ be the \e{southwest} corner of $Q$, then $Q=e_Q+2^{-m}Q_0=e_Q+\ell(Q)Q_0$. Let $Q_0=\cup_{i=1}^{2^n}Q_i^{(0)}$, where the $Q_i^{(0)}$'s are the children of $Q_0$. Select any $2^{n-1}$ children of $Q_0$ for example $\left\{ Q_i^0 \right\}_{i=1}^{2^{n-1}}$. Let $\widetilde Q_0=\cup_{i=1}^{2^{n-1}}Q_i^{(0)}.$ For $Q=e_Q+\ell(Q)Q_0$ a dyadic cube in $Q_0$, let $\widetilde Q=e_Q+\ell(Q)Q_0$.

\begin{lma}\label{lma.0.4}
    Let $\o$ be a doubling measure in $\R^n$. There exists $\alpha_0\in (0,\frac{1}{4})$ depending on $n$ and the doubling constant on $\o$ such that for any dyadic cube $S$ in $\R^n$, if $\left\{ Q_i \right\}_{i=1}^{2^n}$ denote its children and $\widetilde S=\cup_{i=1}^{2^{n-1}}$ then
    \begin{equation}
        \alpha_0\leq \frac{\o(\widetilde S)}{\o(S)}\leq 1-\alpha_0.
        \label{eqn.0.11}
    \end{equation}
\end{lma}
\begin{proof}
    Note that with the notation above for $i=1,\dots,2^n$, $\frac{1}{2}Q_i\ss Q_i\ss S\ss 3Q_i$. Thus since $\o$ is doubling $\o(S)\leq \kappa\o(\frac{1}{2}Q_i)$, where $\kappa>1$ depends on $n$ and the doubling constant of $\o$. Note that $\o(\widetilde S)\geq \o(Q_i)\geq \frac{1}{\kappa}\o(S)$ and $\o(\widetilde S)=\o (S\backslash \cup_{i=2^{n-1}+1}^{2^n}Q_i) \leq \o(S)-\o(\frac{1}{2}Q_{2^n})\leq \left( 1-\frac{1}{\kappa} \right)\o(S)$. This shows \eqref{eqn.0.11} with $\alpha_0=\frac{1}{\kappa}.$
\end{proof}

\begin{rmk}\label{rmk0.1}
    Let $S$ be a dyadic cube and let $x_S$ be its center. $\widetilde S$ can be chosen such that for $0<r<r_0=r_0(n)$, there exist balls $ B_1\ss \widetilde S$ and $B_2\ss S\backslash\widetilde S$
    of radius $r\ell(S)$, and ${\text{ dist }}(x_S, B_i)\simeq r\ell(S)$ with comparison constants depending only on $n$.
    
\end{rmk}

\begin{lma}\label{lma.0.5}
    Let $\o$ be a doubling measure in $\R^n$. Given $ M\gg 1$ there exists $\delta_0>0$ such that if $E\ss \O_0$ with $\o(E)<\delta$ and $\delta\leq \delta_0$ then there is a function $F\in L^\infty(\o)$ such that for all $x\in E$
    \begin{equation}
        \sum_{Q,Q'\in \D_0,x\in Q'\ss Q}\left\lvert \fint_Q Fd\o -\fint_{Q'} Fd\o \right\rvert^2\geq M,
        \label{eqn.0.12}
    \end{equation}
    where $Q'$ is a child of $Q$.
\end{lma}
\begin{proof}
    Let $\varepsilon_0=\varepsilon_0(M)>0$ to be determined later, and let $\delta_0$ be the corresponding quantity from Lemma \ref{lma.0.3}. Assume $\o(E)<\delta_0$ then $E$ has a good $\varepsilon_0$-cover $\left\{ \O_i \right\}_{i=1}^k$ with $k\sim \frac{\log \o(E)}{\log \varepsilon_0}.$ Let $\{ S_i^{(l)} \}_{1\leq l\leq k,i\geq 1}$ be as in Definition \ref{dfn.0.1}.

    Let 
\begin{equation}\label{ppt-u}
\cU_l=\cup_{i\ge 1} S_i^{(l)}
\hbox{  and  }\widetilde\cU_l=\cup_{i\ge 1} \widetilde S_i^{(l)},
\end{equation}
 where $ \widetilde S_i^{(l)}$ is as in Lemma \ref{lma.0.4}. Recall that $\cO_l\subset \cU_l$, $\cU_l\backslash \cO_l\subset\partial Q_0$ and by assumption $\om(\partial Q_0)=0$. 
Note that since $E\cap\partial Q_0$ by Definition \ref{dfn.0.1} we have
\begin{equation}\label{u-inclusion}
E\subseteq \cU_k\subseteq \cU_{k-1}\subseteq \cdots \subseteq \cU_1
\end{equation}
and for all $1\le l\le k$
\begin{equation}\label{u-small}
\omega(\cU_l\cap
S^{(l-1)}_i)\le\varepsilon_0 \omega(S^{(l-1)}_i).
\end{equation}

Define
    \begin{equation}
        F=\sum_{j=2}^k \chi_{\widetilde \U_{j-1}\backslash \U_j}.
        \label{eqn.0.13}
    \end{equation}
    We first remark that the function $F$ only takes values $0$ and $1$. In fact, let $x\in \R^n$, $\chi_{\widetilde \U_{j-1}\bs \cU_j}(x)=1$ if and only if $x\in \wt \U_{j-1}\bs \U_j$. Since $x\notin \U_j$ then $x\notin \U_l$ for any $l=j,\dots,k$ because in this case $\U_l\ss \U_j$. Since $\wt \U_l\ss \U_l$ then $x\notin \wt \U_l$ for $l=j,\dots,k$. In particular for $l=j,\dots,k$, $\chi_{\wt \U_l\bs \U_{l+1}}=0$. 
 On the other hand since $x\in \wt\U_{j-1}$ then $x\in \U_{j-1}\subset\U_l$ for any $l\le j-1$ and  $\chi_{\U_{l-1}\bs \U_l}(x)=0$
    Thus there is at most one $j$ for which $\chi_{\U_{j-1}\bs \U_j}(x)$ is different than $0$. That is $F=\chi_{\{x:F(x)=1\}}.$

    For $x\in E$, as $x\in \O_l$ for all $l=1,\dots,k$, there is $S_i^{(l)}\ni x$. Let $T_i^{(l)}$ be the child of $S_i^{(l)}$ containing $x$. In order to prove \eqref{eqn.0.12} we need to estimate
    \begin{equation}
        \fint_{S_i^{(l)}}Fd\o-\fint_{T_i^{(l)}}Fd\o.
        \label{eqn.0.14}
    \end{equation}
    We consider each term separately.
    \begin{align}
        \fint_{S_i^{(l)}} Fd\o &= \frac{1}{\o(S_i^{(l)})}\sum_{j=2}^k\o\left( ( \wt \U_{j-1}\bs \U_j )\cap S_i^{(l)} \right)\nonumber\\
        &= \frac{1}{\o(S_i^{(l)})}\sum_{j=2}^l \o\left( ( \wt \U_{j-1}\bs \U_j )\cap S_i^{(l)} \right) + \frac{1}{\o(S_i^{(l)})}\sum_{j=l+2}^k \o\left( ( \wt \U_{j-1}\bs \U_j )\cap S_i^{(l)} \right)
        \nonumber\\
        &+ \frac{\o\left( ( \wt \U_{l}\bs \U_{l+1} )\cap S_i^{(l)} \right)}{\o(S_i^{(l)})},
        \label{eqn.0.15}
    \end{align} 
    and
    \begin{align}
        \fint_{T_i^{(l)}} Fd\o &=  \frac{1}{\o(T_i^{(l)})}\sum_{j=2}^k\o\left( ( \wt \U_{j-1}\bs \U_j )\cap T_i^{(l)} \right)\nonumber\\
        &= \frac{1}{\o(T_i^{(l)})}\sum_{j=2}^l \o\left( ( \wt \U_{j-1}\bs \U_j )\cap T_i^{(l)} \right) + \frac{1}{\o(T_i^{(l)})}\sum_{j=l+2}^k \o\left( ( \wt \U_{j-1}\bs \U_j )\cap T_i^{(l)} \right)
        \nonumber\\
        &+ \frac{\o\left( ( \wt \U_{l}\bs \U_{l+1} )\cap T_i^{(l)} \right)}{\o(T_i^{(l)})},
        \label{eqn.0.16}
    \end{align}
    For $2\leq j\leq l$, $T_i^{(l)}\ss S_i^{(l)}\ss \U_l\ss \U_j$ thus 
    \begin{equation*}
    \o\left( ( \wt \U_{j-1}\bs \U_j )\cap S_i^{(l)} \right)=0\hbox{  and  }\o\left( ( \wt \U_{j-1}\bs \U_j )\cap T_i^{(l)} \right)=0.
    \end{equation*} 
    For $l+2\leq j\leq k$ by Lemma \ref{lma.0.2} and the fact that $\o$ is doubling we have
    \begin{eqnarray}
        \o\left( ( \wt \U_{j-1}\bs \U_j )\cap T_i^{(l)} \right)&\leq& \o\left( ( \wt \U_{j-1}\bs \U_j )\cap S_i^{(l)} \right)
        \nonumber\\
        &\leq&\o\left( \wt \U_{j-1}\cap S_i^{(l)} \right)\leq \varepsilon_0^{j-1-l}\o(S_i^{(l)})\leq \varepsilon_0^{j-1-l}\kappa\o(T_i^{(l)}),
        \label{eqn.0.17}
    \end{eqnarray}
    where the constant $\kappa$ only depends on $n$ and the doubling constant of $\o$. Therefore,
    \begin{eqnarray}
        \left\lvert \fint_{S_i^{l}}Fd\o-\fint_{T_i^{(l)}}Fd\o-\left( \frac{\o\left( ( \wt \U_{l}\bs \U_{l+1} )\cap S_i^{(l)} \right)}{\o(S_i^{(l)})} - \frac{\o\left( ( \wt \U_{l}\bs \U_{l+1} )\cap T_i^{(l)} \right)}{\o(T_i^{(l)})}\right)\right\rvert\nonumber\\
        \leq \sum_{j=l+2}^k \varepsilon_0^{j-1-l}+\varepsilon_0^{j-1-l}\kappa\leq (1+\kappa)\frac{\varepsilon_0}{1-\varepsilon_0}.
        \label{eqn.0.18}
    \end{eqnarray}

    Note that since $\wt S_i^{(l)}\ss S_i^{(l)}$,
    \begin{eqnarray}
        \o\left( ( \wt \U_l\bs \U_{l+1} )\cap S_i^{(l)} \right)&=& \o\left( ( \wt S_i^{(l)}\bs \U_{l+1} )\cap S_i^{(l)} \right)
        \nonumber\\
        &=& \o\left( \widetilde S_i^{(l)}\cap \U_{l+1}^c \right)=\o\left( \wt S_i^{(l)} \right)-\o\left( \wt S_i^{(l)}\cap \U_{l+1} \right).
        \label{eqn.0.19}
    \end{eqnarray}

    Hence since for every $j$, $\o( \U_j\bs \O_j )=0$ by Lemma \ref{lma.0.1},
    \begin{equation}
        \left\lvert \o\left( (\wt \U_l\bs\U_{l+1}) \cap S_i^{l}\right)-\o\left(\wt S_i^{(l)}\right)\right\rvert \leq \o\left( S_i^{(l)}\cap \U_{l+1} \right) = \o\left(S_i^{(l)}\cap \O_{l+1}\right) \leq \varepsilon_0\o\left( S_i^{(l)} \right).
        \label{eqn.0.20}
    \end{equation}
    Combining \eqref{eqn.0.11} with \eqref{eqn.0.20} we obtain
    \begin{eqnarray}
        \o\left( \wt S_i^{(l)} \right)-\varepsilon_0\o\left(S_i^{(l)}\right) &\leq \o\left( \wt \U_l\bs \U_{l+1}\cap S_i^{(l)} \right) &\leq \o\left(\wt S_i^{(l)}\right)+\varepsilon_0\o\left( S_i^{(l)} \right)\nonumber\\
        (\alpha_0-\varepsilon_0)\,\o\left(S_i^{(l)} \right) &\leq \o\left( \wt \U_l\bs \U_{l+1}\cap S_i^{(l)} \right) &\leq (1-\alpha_0+\varepsilon_0)\,\o\left( S_i^{(l)} \right)
        \label{eqn.0.21}
    \end{eqnarray}
    To estimate $\o\left( \wt \U_l\bs \U_{l+1}\cap T_i^{(l)} \right)$, we need to consider two cases, either $T_i^{(l)}\ss \wt S_i^{(l)}$ or not. If $T_i^{(l)}\ss \wt S_i^{(l)}$ then as in \eqref{eqn.0.19} and \eqref{eqn.0.21} and by the doubling properties of $\o$ we have
    \begin{equation}
        \o\left( \wt \U_l\bs \U_{l+1}\cap T_i^{(l)} \right)=\o\left( \wt S_i^{(l)}\cap T_i^{(l)} \right)-\o \left( T_i^{(l)}\cap \U_{l+1} \right).
        \label{eqn.0.22}
    \end{equation}
    and
    \begin{align}
        \left\lvert \o\left(\wt \U_l\bs \U_{l+1} \cap T_i^{(l)}\right)-\o\left( T_i^{(l)} \right)\right\rvert &\leq \o \left( T_i^{(l)}\cap \U_{l+1} \right)\nonumber\\
        &= \o \left( T_i^{(l)}\cap \O_{l+1} \right) \leq \o\left( S_i^{(l)}\cap \O_{l+1} \right)\leq \varepsilon_0\,\o\left( S_i^{(l)} \right)\nonumber\\
        &\leq \kappa\varepsilon_0\,\o(T_i^{(l)}),
        \label{eqn.0.23}
    \end{align}
    which implies
    \begin{equation}
        \left( 1-\kappa\varepsilon_0 \right)\o\left( T_i^{(l)} \right)\leq \o\left(\wt \U_l\bs \U_{l+1} \cap T_i^{(l)}\right)\leq \o\left( T_i^{(l)} \right).
        \label{eqn.0.24}
    \end{equation}
    Combining \eqref{eqn.0.21} and \eqref{eqn.0.24} we have
    \begin{equation}
        -1+\alpha_0-\varepsilon_0\leq\frac{\o\left( ( \wt\U_l\bs \U_{l+1} )\cap S_i^{(l)} \right)}{\o\left( S_i^{(l)} \right)} - \frac{\o\left(( \wt\U_l\bs \U_{l+1} \cap T_i^{(l)} \right)}{\o\left( T_i^{(l)} \right)}
        \leq \left( \kappa+1 \right)\varepsilon_0-\alpha_0
        \label{eqn.0.25}
    \end{equation}
    Combining \eqref{eqn.0.18} and \eqref{eqn.0.25} we have that
    \begin{equation}
        \left\lvert \fint_{S_i^{(l)}}Fd\o - \fint_{T_i^{(l)}}Fd\o\right\rvert \geq \min\left\{ 1-\alpha_0+\varepsilon_0, \alpha_0-(\kappa+1)\varepsilon_0 \right\}-(1+k)\frac{\varepsilon_0}{1-\varepsilon_0}.
        \label{eqn.0.26}
    \end{equation}
    In the case $T_i^{(l)}\not\subset \wt S_i^{(l)}$ then $T_i^{(l)} \cap\ \wt \U_l\bs \U_{l+1}\ss \partial T_i^{(l)}$, and by Lemma \ref{lma.0.1}, $\o\left( T_i^{(l)}\cap \wt \U_l\bs \U_{l+1} \right)=0$. This combined with \eqref{eqn.0.21} yields
    \begin{equation}
        \alpha_0-\varepsilon_0\leq\frac{\o\left( ( \wt\U_l\bs \U_{l+1} )\cap S_i^{(l)} \right)}{\o\left( S_i^{(l)} \right)} - \frac{\o\left( ( \wt\U_l\bs \U_{l+1} )\cap T_i^{(l)} \right)}{\o\left( T_i^{(l)} \right)}
        \leq 1-\alpha_0+\varepsilon_0
        \label{eqn.0.27}
    \end{equation}
    Combining \eqref{eqn.0.18} and \eqref{eqn.0.27} we have that
    \begin{equation}
        \left\lvert \fint_{S_i^{(l)}}Fd\o - \fint_{T_i^{(l)}}Fd\o\right\rvert \geq \min\left\{ \alpha_0-\varepsilon_0, 1-\alpha_0+\varepsilon_0 \right\}-(1+k)\frac{\varepsilon_0}{1-\varepsilon_0}.
        \label{eqn.0.29}
    \end{equation}
    Choose $\varepsilon_0$ small enough so that in both cases, that \eqref{eqn.0.26} and \eqref{eqn.0.29}
    \begin{equation}
        \left\lvert \fint_{S_i^{(l)}}Fd\o - \fint_{T_i^{(l)}}Fd\o\right\rvert \geq \min\left\{ \frac{\alpha_0}{2}, 1-2\alpha_0 \right\}=\beta_0.
        \label{eqn.0.30}
    \end{equation}
    For $x\in E$ since $x\in\O_l$ for all $l=1,\dots,k$ and for $k\geq 2$
    \begin{equation}
        \sum_{Q,Q'\in \D_0,x\in Q'\ss Q}\left\lvert \fint_Q Fd\o -\fint_{Q'} Fd\o \right\rvert^2\geq (k-1)\beta_0^2\geq \frac{k\beta_0^2}{2}\geq M
        \label{eqn.0.31}
    \end{equation}
    choosing $\delta\leq \delta_0$ so that $k\sim \frac{\log \o(E)}{\log \varepsilon_0} \geq \frac{\log \delta_0}{\log \varepsilon_0}\geq\frac{2M}{\beta_0^2}.$
\end{proof}

\section{The domain above the graph of a Lipschitz function}\label{sec.1}
\setcounter{equation}{0}

We start the section working on the upper half plane $\R_+^n=\left\{ (x,t):t>0 \right\}$, with boundary $\R^{n-1}$, $x\in \R^{n-1}$. We will denote points in $\R_+^n$ by capital letters. Let $S\ss \R^{n-1}$ be a cube. Let $A_S=(x_S, \ell(S))\in \R_+^n$, where $x_S$ is the center of $S$ and $\ell(S)$ is the side-length of $S$. For $v=\left( v_1,\dots,v_n \right)\in \R^{n-1}$, $\lVert v\rVert = \max_i \lvert v_i\rvert$, and $\lvert v\rvert=\left( \sum_{i=1}^{n-1}\lvert v_i\rvert^2 \right)^{\frac12}.$ We then have $S=\left\{ x\in \R^{n-1}: \lVert x-x_S\rVert<{\ell(S)}/{2} \right\}$. Recall that
\begin{equation}\label{TS}
T(S)=\left\{ X=(x,t), x\in S, 0\leq t\leq \ell(S) \right\}
\end{equation}
and for $\eta>0$
\begin{equation}\label{TSeta}
T_\eta{S}=\left\{ X=(x,t): x\in S, \frac{\eta}{2}\ell(S)\leq t\leq \ell(S) \right\}
\hbox{ and } A_\eta(S)=\left( x_S, \frac{\eta}{2}\ell(S) \right).
\end{equation}
For $\gamma>0$, $h>0$ and $x\in \R^{n-1}$, let 
\[
\Gamma_\gamma(x,0)=\left\{ Y=(y,s)\in \R_+^n: \lvert y-x\rvert \leq \gamma s \right\}
\]
and 
\[
\Gamma_\gamma^h(x,0)=\left\{ Y\in \Gamma_\gamma(x,0), Y=(y,s): 0<s\leq h \right\}.
\]
Recall that $L=\div A(X)\nabla$ denotes a divergence form elliptic operator in $\R^n_+$, where $A(X)=\left( a_{ij}(X) \right)_{i,j=1}^n$ is a real $n\times n$ matrix with $a_{ij}\in L^\infty$ satisfying the uniform ellipticity condition, but not necessarily symmetric, and that $\o^X$ is the elliptic measure for $L$ with pole at $X$.
The main estimate in this section is the following:
\begin{prop}
    Fix a cube $Q\ss \R^{n-1}$, and the corresponding point $A_Q=(x_Q,\ell(Q))$. Then, if $E\ss Q$ is a Borel set, there are constants $\delta_0>0, \gamma>0$, and $C_1>0$, depending only on dimension and ellipticity, and a Borel set $H\ss Q$, such that if $\o^{A_Q }(E)<\delta_0$ and $u$ is the solution to
     \eqref{eqn.dp}, with $F=\chi_H$ then
    \begin{equation}
        C_1\ol k^{\frac12}\leq S_\gamma^{\gamma \ell(Q)}(u)(x,0) \text{  for all  }(x,0)\in E.
        \label{eqn.1.2}
    \end{equation}
    Here 
   \begin{equation}\label{square-funct}
   S_\gamma^{\gamma \ell(Q)}(u)(x,0)=\left( \int_{\Gamma_\gamma^{\gamma \ell(Q)}(x,0)}\lvert \nabla u(Y)\rvert^2\lvert Y-(x,0)\rvert^{2-n}dY\right)^{\frac12}
   \end{equation} and 
   \begin{equation}\label{kbar}
   \ol k=\left[ \log \frac{1}{\o^{A_Q}(E)} \right]\, .
   \end{equation}
    \label{prop.1.1}
\end{prop}

The proof of Proposition \ref{prop.1.1} requires several lemmas. To simplify the notation we first assume that $Q=Q_0$. We set $F$ as in \eqref{eqn.0.13}, with $\o=\o^{A_{Q_0}}$ and $H=\left\{x\in \R^{n-1} F(x)=1 \right\}$ as in Lemma \ref{lma.0.5}, so that $F=\chi_H$. 
By the maximum principle $0\leq u\leq 1$. We now use the construction in Section \ref{sec.0}. Note that if $x\in E$ then $x\in \U_l, l=1,\dots,k$. Fix such an $l$, there exists a unique $S_i^{(l)}\ss\U_l$ such that $x\in S_i^{(l)}$ (possibly removing points in $E\cap\cup\s\ss \cup_l\partial\s$ a subset of both $\o$ and Lebesgue measure zero (see Lemma \ref{lma.0.1})). Moreover there exists a unique $\wh S_i^{(l)}$, a child of $S_i^{(l)}$, such that $x\in \wh S_i^{(l)}$.

\begin{lma}\label{lma.1.3}
    There exist $a>0$, $\eta >0$, $\varepsilon_0>0$  and a corresponding $\delta_0>0$ depending only on the ellipticity and the dimension, so that if $E$ is as in Lemma \ref{lma.0.3}, $F$ is a above and $u$ is the solution 
     \eqref{eqn.dp}, with $F=\chi_H$ then
    \begin{equation}
        \left\lvert u\left(A_\eta\left( S_i^{(l)} \right)\right)-u\left( A_\eta \left( \wh S_i^{(l)} \right)\right) \right\rvert \geq a.
        \label{eqn.1.4}
    \end{equation}
    
\end{lma}
\begin{proof}
    Let $K\left(X,(y,0)\right)=K\left( X,y \right), X\in \R_+^n,\ y\in \R^{n-1}$ be the kernel function (see \cite{KKPT}, (1.15)). Then,
    \begin{align}\label{eqn1.5A}
        u\left( A_\eta (S_i^{(l)}) \right)&=\int_{S_i^{(l)}}K\left( A_\eta(S_i^{(l)}),y \right)F(y)d\o(y)\\
        &+ \int_{^cS_i^{(l)}}K\left( A_\eta (S_i^{(l)}),y \right)F(y)d\o(y).\nonumber
    \end{align}
    We  first estimate the second term. Let $r_{i,l}=\ell(S_i^{(l)})$. By the boundary H\"older estimate ((1.9) in \cite{KKPT}), since $\lvert A_\eta(S_i^{(l)})-x_{S_i^{(l)}}\rvert = \frac{\eta}{2}r_{i,l}$,
    we obtain that
    \begin{equation}\label{eqn1.5B}
        \int_{^cS_i^{(l)}}K\left( A_\eta\left( S_i^{(l)} \right),y \right)F(y)d\o(y) \leq C\left[ \frac{\lvert A_\eta(S_i^{(l)})-x_{S_i^{(l)}}|}{r_{i,l}} \right]^\beta\le C\eta^\beta
    \end{equation}
    where $C$ and $ \beta$ depend only on the ellipticity and $n$. To handle the first term, we recall from \cite{KKPT} (1.15) and the doubling property of $\o$ (a consequence of (1.13) in \cite{KKPT} and the Harnack principle (1.16) in \cite{KKPT}), that
    \begin{equation}
        K\left( A_\eta(S_i^{(l)}),y \right)\simeq \frac{1}{\o(S_i^{(l)})}, \text{ for }y\in S_i^{(l)}
        \label{eqn.1.5}
    \end{equation}
    with comparability constants depending on the dimension, the ellipticity and $\eta$.

     From the definition of $F$, the first term equals:
    \begin{align*}
        \tag{I}\label{eqn.I} &\sum_{j=2}^l \int_{S_i^{(l)}}K\left( A_\eta ( S_i^{(l)} ),y \right)\chi_{\wt \U_{j-1}\bs \U_j}(y)d\o(y)\\
        \tag{II}\label{eqn.II} +&\sum_{j=l+2}^k \int_{S_i^{(l)}}K\left( A_\eta ( S_i^{(l)} ),y \right)\chi_{\wt \U_{j-1}\bs \U_j}(y)d\o(y)\\
        \tag{III}\label{eqn.III} +&\int_{S_i^{(l)}}K\left( A_\eta ( S_i^{(l)} ),y \right)\chi_{\wt \U_{l}\bs \U_{l+1}}(y)d\o(y)\\
        &= \text{\ref{eqn.I} + \ref{eqn.II} + \ref{eqn.III}}
    \end{align*}
    Since $S_i^{(l)}\ss \U_l\ss \U_j$ for $ 2\leq j\leq l$ then \ref{eqn.I}$=0$. To estimate \ref{eqn.II}, for $l+2\leq j\leq k$, note that by Lemma \ref{lma.0.2}
    \begin{eqnarray}\label{est2}
    \o\left( ( \wt \U_{j-1}\bs \U_j )\cap S_i^{(l)} \right) &\leq& \o\left( \wt \U_{j-1}\cap S_i^{(l)} \right)\\
    &\leq &\o(\U_{j-1}\cap\s)\leq\varepsilon_0^{j-1-l}\o\left( S_i^j \right).
    \end{eqnarray}
    
     Hence for $\varepsilon_0$ small, using \eqref{eqn.1.5} and \eqref{est2} we have
     \begin{eqnarray}\label{eqn1.5C}
     \ref{eqn.II}\leq C(\eta)\sum_{j=l+2}^k \varepsilon_0^{j-1-l}\leq \wt C(\eta)\varepsilon_0.
     \end{eqnarray}
      Thus by \eqref{eqn1.5A}, \eqref{eqn1.5B}, \eqref{eqn.I}, \eqref{eqn.II}, \eqref{eqn.III} and \eqref{eqn1.5C} we have
    \begin{equation}\label{eqn1.5D}
        \left\lvert u\left(A_\eta ( S_i^{(l)} )\right)-\int_{S_i^{(l)}}K\left( A_\eta( S_i^{(l)} ),y \right)\chi_{\wt \U_l\bs \U_{l+1}}(y)d\o(y)\right\rvert \leq C\eta^\beta+\wt C(\eta)\varepsilon_0.
    \end{equation}

    We next consider the second term inside the absolute value. Recall from Section \ref{sec.0} \eqref{ppt-u} that $\wt S_i^{(l)}\ss S_i^{(l)}\ss \U_l$ and $ \U_{l+1}\ss \U_l$. Then $S_i^{(l)}\cap \left( \wt \U_l\bs \U_{l+1} \right) = \wt S_i^{(l)} \bs \left( \wt S_i^{(l)} \cap \U_{l+1} \right)$. Then
    \begin{eqnarray}\label{eqn1.5E}
    \int_{S_i^{(l)}}K\left( A_\eta( S_i^{(l)} ),y \right)\chi_{\wt \U_l\bs \U_{l+1}}(y)d\o(y)&=&\nonumber\\
    \int_{\wt S_i^{(l)}}K\left( A_\eta \left( S_i^{(l)} \right),y \right)d\o(y) &-& \int_{\wt S_i^{(l)}\cap \O_{l+1}}K\left( A_\eta \left( S_i^{(l)} \right),y \right)d\o(y).
    \end{eqnarray}
     Arguing as in the treatment of \ref{eqn.II} above (see \eqref{eqn.1.5}), we see by \eqref{u-small} that
    \begin{equation}\label{eqn1.5''}
  \left| \int_{\wt S_i^{(l)}\cap \U_{l+1}}K\left( A_\eta \left( S_i^{(l)} \right),y \right)d\o(y)\right|\leq  \frac{C(\eta)}{\o(S_i^{(l)})}\cdot \o\left( \wt S_i^{(l)}\cap \U_{l+1} \right)\leq C(\eta)\varepsilon_0.
 \end{equation}
  In order to further analyze the first term we use Lemma \ref{lma.1.6} below, which says that there exist $\eta_0$ small, and $\alpha_1, 0<\alpha_1<1$, depending only on ellipticity and dimension, such that, for all $0<\eta<\eta_0$, we have
    \begin{equation}
        \alpha_1\leq \int_{\wt S_i^{(l)}}K\left( A_\eta\left( S_i^{(l)} \right),y \right)d\o(y) \leq 1-\alpha_1.
        \label{eqn.1.5'}
    \end{equation}
    We continue with the proof of Lemma \ref{lma.1.3} assuming that \eqref{eqn.1.5'} holds. Using \eqref{eqn1.5D}, \eqref{eqn1.5E}, \eqref{eqn1.5''} and \eqref{eqn.1.5'} we obtain
    \begin{equation}\label{eqn1.5F}
        u\left( A_\eta \left( S_i^{(l)} \right) \right)=\int_{\wt S_i^{(l)}} K\left( A_\eta \left( S_i^{(l)} \right),y \right)d\o(y)+e_1,
    \end{equation}
    where $|e_1|\leq C_1\eta^\beta+\wt C_1(\eta)\varepsilon_0$. A similar argument shows that
    \begin{equation}\label{eqn1.5G}
        u\left( A_\eta \left( \wh S_i^{(l)} \right) \right)=\int_{\widehat S_i^{(l)}}K\left( A_\eta \left( \wh S_i^{(l)} \right),y \right)\chi_{\wt \U_l\bs \U_{l+1}}(y)d\o(y)+e_2,
    \end{equation}
    where $\lvert e_2\rvert \leq C_2\eta^\beta+\wt C_2(\eta)\varepsilon_0$.

    Next, with $\alpha_1, \eta_0$ as in \eqref{eqn.1.5'} above, choose $\eta_1\leq \eta_0$ so small that if $0<\eta\leq \eta_1$, $C_1\eta^\beta\leq \frac{\alpha_1}{2}$, and $C_2\eta^\beta\leq \frac{\alpha_1}{16}$. We first consider the quantity $\int_{\wh S_i^{(l)}}K\left( A_\eta\left( \widehat S_i^{(l)} \right),y \right)d\o(y)$. Applying the boundary H\"older estimate ((1.9) in \cite{KKPT}) to the solution $1-\int_{\wh S_i^{(l)}}K(X,y)d\o(y)$ in the box $T\left(\wh S_i^{(l)}\right)$, we obtain that 
 \begin{equation}\label{eqn1.5H}
 \int_{\wh S_i^{(l)}}K\left( A_\eta \left( \wh S_i^{(l)},y \right)d\o(y) \right)=1-e_3,
 \end{equation}
  where $0\leq e_3\leq C_3\eta^\beta$. Choose now $\eta_2\leq \eta_1$ so that for $0<\eta_2 \leq \eta_1$, we have $0\le e_3\leq \frac{\alpha_1}{32}$. We next turn to estimating
  \begin{equation}\label{eqn1.5I}
  \int_{\wh S_i^{(l)}}K\left( A_\eta \left( S_i^{(l)} \right),y \right)\chi_{\widetilde \U_l\bs \U_{l+1}}(y) d\o(y).
  \end{equation}

    Note that $\wh S_i^{(l)}\cap \,\widetilde\U_l\bs \U_{l+1}\subset\wh S_i^{(l)}\cap \wt S_i^{(l)}\cap\,\U_{l+1}^C=\left[ \wh S_i^{(l)}\cap \wt S_i^{(l)} \right] \bs \left[ \wh S_i^{(l)}\cap \wt S_i^{(l)}\cap \U_{l+1} \right]$. We consider two cases:

\bigskip
{\bf Case 1:} $\wh S_i^{(l)}\cap \wt S_i^{(l)}=\emptyset$, that is  $\wh S_i^{(l)}$ was not one of the children chosen to form $\widetilde S_i^{(l)}$. In this case the integral in \eqref{eqn1.5I} is $0$ and hence by \eqref{eqn1.5G}
            \begin{equation}\label{eqn1.5J}
                u\left( A_\eta( \wh S_i^{(l)} ) \right)=e_2.
            \end{equation}

{\bf Case 2:} $\wh S_i^{(l)}\cap \wt S_i^{(l)}=\wh S_i^{(l)}$, that is  $\wh S_i^{(l)}$ was chosen to form $\wt S_i^{(l)}$. In this case, we have by \eqref{eqn1.5G}, \eqref{eqn1.5H}, \eqref{eqn.1.5} and (iii) in Definition \ref{dfn.0.1} that
            \begin{align}\label{eqn1.5K}
                u\left( A_\eta( \wh S_i^{(l)} ) \right)&= \int_{\wh S_i^{(l)}}K\left( A_\eta( \wh S_i^{(l)} ),y \right)\chi_{\wt \U_l\bs \U_{l+1}}(y)d\o(y)+e_2\nonumber\\
                &= \int_{\wh S_i^{(l)}}K\left( A_\eta( \wh S_i^{(l)}),y \right)d\o(y) - \int_{\wh S_i^{(l)}\cap\, \U_{l+1}}K\left( A_\eta( \wh S_i^{(l)} ),y \right)d\o(y)+e_2\nonumber\\
                &= 1-e_3+\wt C_3(\eta)\varepsilon_0+e_2.
            \end{align}
 In order to control $e_1$, $e_2$, $e_3$ and the additional error in \eqref{eqn1.5K} choose $\varepsilon_0$ small enough so that
            \begin{equation*}
                \lvert \wt C_1(\eta_2)\varepsilon_0\rvert \leq \frac{\alpha_1}{4}\qquad \lvert \wt C_2(\eta_2)\varepsilon_0\rvert \leq \frac{\alpha_1}{16}\qquad \lvert \wt C_3\left( \eta_2 \right)\varepsilon_0\rvert \leq \frac{\alpha_1}{32},
            \end{equation*}
            and from now on fix $\eta=\eta_2$ and $\varepsilon_0$. Note that the choice of $\eta_2, \varepsilon_0$ depends only on ellipticity and dimension.  Thus \eqref{eqn1.5F} becomes
     \begin{equation}\label{eqn1.5F1}
       \left| u\left( A_{\eta_2} ( S_i^{(l)} ) \right)-\int_{\wt S_i^{(l)}} K( A_{\eta_2}\left( S_i^{(l)} ),y \right)d\o(y)\right|\leq |e_1|\le C_1\eta_2^\beta + \wt C_1(\eta_2)\varepsilon_0\le \frac{\alpha_1}{2} +\frac{\alpha_1}{4}
    \end{equation}
    which combined with \eqref{eqn.1.5'} yields
   \begin{equation}\label{eqn1.5L}
 \frac{\alpha_1}{4}\leq  u\left( A_{\eta_2} ( S_i^{(l)} ) \right)\leq 1-\frac{\alpha_1}{4}.
   \end{equation}

    Consider now $u\left( A_{\eta_2}(\wh S_i^{(l)} ) \right)$. In Case 1, we show that $u\left( A_{\eta_2}( \wh S_i^{(l)}) \right)=e_2$ where $\lvert e_2\rvert \leq C_2\eta_2^\beta+\lvert \wt C_2( \eta_2 )\varepsilon_0\rvert \leq \frac{\alpha_1}{16}+\frac{\alpha_1}{16}=\frac{\alpha_1}{8}$. In Case 2, we have $u\left( A_{\eta_2}( \wh S_i^{(l)} ) \right)=1-e_3+\wt C_3(\eta_2)\varepsilon_0+e_2$, and $0\leq e_3\leq \frac{\alpha_1}{32}$, $\lvert e_2\rvert \leq \frac{\alpha_1}{16}+\frac{\alpha_1}{16}, \lvert\wt C_3(\eta_2)\varepsilon_0\rvert \leq \frac{\alpha_1}{32}$. Thus in Case 2, we have:
    \begin{equation}\label{eqn1.5M}
        1-\frac{\alpha_1}{16}-\frac{\alpha_1}{8}\leq u\left( A_{\eta_2}( \wh S_i^{(l)}  ) \right)\leq 1.
    \end{equation}
  Hence, in Case 1 by \eqref{eqn1.5J} and \eqref{eqn1.5L} we have
\begin{equation*}  
 u\left( A_{\eta_2}( S_i^{(l)}  \right) )-u\left( A_{\eta_2}( \wh S_i^{(l)}) \right)\geq \frac{\alpha_1}{4}-\frac{\alpha_1}{8}\geq \frac{\alpha_1}{8},
 \end{equation*}
 while in Case 2 by \eqref{eqn1.5K} and \eqref{eqn1.5L} we have
    \begin{equation*}
        u\left( A_{\eta_2}( \wh  S_i^{(l)}) \right)-u\left( A_{\eta_2}( S_i^{(l)}) \right)\geq 1-\frac{\alpha_1}{16}-\frac{\alpha_1}{8}-1+\frac{\alpha_1}{4}=\frac{\alpha_1}{16}, 
    \end{equation*}
    and Lemma \ref{lma.1.3} follows (once we have established Lemma \ref{lma.1.6} below).

\end{proof}

\begin{lma}\label{lma.1.6}
    There exist $\eta_0>0$, and $\alpha_1\in (0,1)$, both depending only on the ellipticity and the dimension, so that, with the notations above, we have for all $0<\eta<\eta_0$,
    \begin{equation*}
        \alpha_1\leq \int_{\wt S_i^{(l)}}K\left( A_\eta(S_i^{(l)}),y \right)d\o(y)\leq 1-\alpha_1.
    \end{equation*}
\end{lma}
\begin{proof}
    By translation and scaling, we can assume $\s = Q_0$, $\wt S_i^{(l)}=\wt Q_0$, $ A_\eta=A_\eta(Q_0)=(x_{Q_0}, \eta/2)$ (see \eqref{TSeta}), and prove that if $u$ is the solution to
    \begin{equation*}
        \left\{
        \begin{aligned}
            Lu&= 0 \text{ in }\R_+^n\\
            u\vert_{ \R^{n-1}}&= \chi_{\wt Q_0}
        \end{aligned}
        \right.
    \end{equation*}
  it satisfies $\alpha_1\leq u(A_\eta)\leq 1-\alpha_1$ for $0<\eta\leq \eta_0$. Consider $0<\eta<\eta_0$, $\eta_0\leq r_0(n)<1$, the constant in Remark \ref{rmk0.1}. Consider the ball $B_1\ss \wt Q_0$ with radius $\eta$, as in Remark \ref{rmk0.1}. Let $x_1$ be the center of $B_1$, and let $A=\left( x_1,\frac{\eta}{M} \right)$, where $M$ will be chosen depending on the ellipticity and the dimension. Let $v(y,t)=1-u(y,t)$, so that $0\leq v\leq 1$. We recall that $\lvert x_1-x_{Q_0}\rvert \simeq \eta$, with comparability constants depending only on dimension. Apply the boundary H\"older continuity estimate  ((1.9) in \cite{KKPT}) to $v$ in the region $T(B_1)$ (see \eqref{TS}). Then,
    \begin{equation*}
        v(A)\leq C\left( \frac{\lvert A-(x_1,0)\rvert}{\eta} \right)^\beta\leq C\left( \frac{1}{M} \right)^\beta.
    \end{equation*}
    We now choose $M$, depending on ellipticity and dimension, so that $C\left( \frac{1}{M} \right)^\beta\leq \frac{1}{2}$. Then, $u(A)\geq \frac{1}{2}$. Using Harnack's principle ( (1.6) in \cite{KKPT}) on a chain of at most $2M$ balls of radius $\frac{\eta}{M}$, we see that $u(x,\eta/2)\geq C\frac{M}{2}$. Hence, since $\lvert x_1-x_{Q_0}\rvert \simeq \eta$, applying Harnack's principle once again, we obtain, that
    \begin{equation*}
        u(x_{Q_0},\eta/2)\geq \frac{\wt C^M}{2},
    \end{equation*}
    which is the desired lower bound. For the upper bound, we consider $u$ in the ball $B_2\ss Q_0\bs \wt Q_0$, with center $x_2$ (see Remark \ref{rmk0.1}). A similar argument, applied to $u$ instead of $v$ yields $u\left( x_2, {\eta}/{2} \right) \leq \frac{1}{2}$, or $v\left( x_2,{\eta}/2 \right)\geq \frac{1}{2}$. Applying Harnack to $v$, in a similar way, gives $v\left( x_{Q_0},\eta/2 \right)\geq \frac{\wt C^M}{2}$ or $1-\frac{\wt C^M}{2}\geq u(x_{Q_0},\eta/2)$, which concludes the proof.
\end{proof}

    We now turn to the proof of Proposition \ref{prop.1.1}. First note that by translation and scale invariance of the hypothesis and the conclusion, it suffices to prove it for the case when $Q=Q_0$.

    \begin{proof}[Proof of Proposition \ref{prop.1.1}]
        Let
        \begin{equation}\label{defq}
            Q_{\eta,i}^l=\left\{ (y,t): \lVert y-x_i^l\rVert \leq \frac{1}{2}r_{i,l}+\frac{\eta}{4}r_{i,l},\  \frac{\eta}{8}r_{i,l}\leq t\leq r_{i,l}+\frac{\eta}{4}r_{i,l} \right\},
        \end{equation}
        where $r_{i,l}=\ell\left( \s \right)$ and $x_i^l$ is the center of $\s$. Let $\eta$ be as in Lemma \ref{lma.1.3}. Let $c_i^l=\fint_{Q_{\eta,i}^l}u$. Using local $L^\infty$ estimates (see (1.4), (1.5) in \cite{KKPT}), Poincar\'e's inequality and the fact that for $(x,t) \in  Q_{\eta,i}^l$, $t\sim r_{i,l}$ then for $u$ as in Lemma \ref{lma.1.3}, we have
        \begin{align}\label{eqn1.6}
            \left\lvert u\left(A_\eta\left( \s \right)\right)-u\left( A_\eta\left(\wh S_i^{(l)} \right)\right)\right\rvert &\leq \left\lvert u\left( A_\eta \left(\s\right) \right) - c_i^l\right\rvert + \left\lvert u\left( A_\eta \left( \wh S_i^{(l)}\right) \right)-c_i^l\right\rvert\nonumber\\
            &\leq C\left( \fint_{Q_{\eta,i}^l} \left\lvert u-c_i^l\right\rvert^2 \right)^{\frac12} + C\left( \fint_{Q_{\eta,i}^l} \left\lvert c-c_i^l\right\rvert^2 \right)^{\frac12}\nonumber \\
            &\leq Cr_{i,l}\left( \fint_{Q_{\eta,i}^l}\lvert \nabla u\rvert^2 \right)^{\frac12}\nonumber\\
            &\leq C\left( \int_{Q_{\eta,i}^l}\lvert \nabla u\rvert^2 \frac{dy dt}{t^{n-2}} \right)^{\frac12}.
        \end{align}
         Here all the constants $C$ above depend only on ellipticity and dimension. If we now apply Lemma \ref{lma.1.3}, \eqref{eqn1.6} yields
        \begin{equation}
            \left( \int_{Q_{\eta,i}^l}\lvert \nabla u\rvert^2 dy\frac{dt}{t^{n-2}} \right) \geq \alpha_0,
            \label{eqn.1.7}
        \end{equation}
        where $\alpha_0$ depends only on the ellipticity and the dimension.

        Note that if $x\in E$, for each $l$ there exists $i$ such that $x\in S_i^{(l)}$, and so, for all $(y,t)\in Q_{\eta,i}^l, \lvert x-y\rvert +t\simeq r_{i,l}\simeq t$, with comparability constants depending only on ellipticity and dimension (since $\eta$ only on the ellipticity and the dimension). Thus, there exists $\gamma$, depending only on ellipticity and dimension, so that $Q_{\eta,i}^l\ss \Gamma_\gamma^\gamma(x,0)$. Also, if $x\in E, x\in \s$ and $x\in S_j^{(l')}$, with $l'<l$, then $2^{l-l'}\ell( S_j^{(l)} )\leq \ell( S_j^{(l')}  )$, by the proper containment of $\s$ into $S_{i'}^{(l-1)}$, remarked in the statement after (iii) in Definition \ref{dfn.0.1}. 
        Note that by choosing $m\in\N$ such that $2^m>\frac{8}{\eta} \left( 1+\frac{\eta}{4} \right)$ then if $l-l'\ge m$ we have $Q_{\eta,i}^l \cap Q_{\eta,j}^{l'}=\emptyset$. In fact in this case
        $( 1+\frac{\eta}{4})r_{i,l}< 2^{-(l-l')}( 1+\frac{\eta}{4})r_{i,l'}\le 2{-m}( 1+\frac{\eta}{4})r_{i,l'}\le\frac{\eta}{8}r_{j,l'}$ (see \eqref{defq}).

        Using this fact, and adding in $l$, we obtain
        \begin{equation*}
            \int_{\Gamma_\gamma^\gamma(x,0)}\lvert \nabla u(y,t)\rvert^2 \frac{dydt}{t^{n-2}} \geq \frac{a_0k}{m},
        \end{equation*}
        $k\sim  \frac{\log \o^{A_{Q_0}}}{\log \varepsilon_0}$ is as in Lemma \ref{lma.0.3}, and $\varepsilon_0$ a in Lemma \ref{lma.1.3}. Proposition \ref{prop.1.1} follows.
    \end{proof}

    We next give an extension of Proposition \ref{prop.1.1} to the domain above the graph of a Lipschitz function 
   \begin{equation}\label{lip}
   D_\varphi = \left\{ (x,t): \\t>\varphi(x), x\in \R^{n-1}, \varphi:\R^{n-1}\to \R, \quad\lVert \nabla \varphi\rVert_\infty = M \right\}.
   \end{equation}  
   If $Y=(x,\varphi(x))\in \partial D_\varphi$, we define the nontangential region 
   \begin{equation*}
   \wt \Gamma_\gamma(Y)=\left\{ x\in D_\varphi: \lvert X-Y\rvert\leq (1+\gamma)\,\text{dist}\left( X,\partial D_\varphi \right) \right\},
   \end{equation*}
    where $\gamma>0$. Note that if $X\in \wt \Gamma_\gamma(Y)$, then dist $(X, \partial D_\varphi)\simeq \lvert X-Y\rvert$, with comparability constant depending only on $\gamma$. We also define truncated non-tangential regions $\wt \Gamma_\gamma^\gamma (Y)$ as follows $\wt \Gamma_\gamma^\tau (Y)=\wt \Gamma^\gamma (Y)\cap B(Y,\tau)$. Note that the bi-Lipschitz change of variable 
    \begin{equation}\label{lip-dom}
    \Phi: \ol D_\varphi\to \ol \R_+^n\hbox{   given by   } \Phi(x,t)= (x,t-\varphi(x))
    \end{equation} 
    preserves the class of operators $L$ under consideration as well as the non-tangential regions, thus we obtain the following result.
   
    \begin{cor}\label{cor.1.8}
    Given $D_\varphi\subset \R^n$ as above there are positive constants
        $\delta_0$, $\gamma>0$, and  $C_1$ depending only on ellipticity, dimension and $M$ the Lipschitz constant of $\varphi$ such that for any
    cube $Q\ss \R^{n-1}$ if $\Delta=\Phi^{-1}(Q)$, $A_\Delta=\Phi^{-1}(A_Q)$ and $E\ss \Delta$ is a Borel set with $\o(E)<\delta_0$, where $\o=\o^{A_\Delta}$ is the elliptic measure associated to $L$ in $D_\varphi$ then there exists a Borel set $H\ss \Delta$ such that the solution $u$ of
    \begin{equation*}
         \tag{$DP_{D_\varphi}$}
        \label{eqn.dp'}
        \left\{
        \begin{aligned}
            Lu&= 0 \text{ in }D_\varphi\\
            u\vert_{\partial D_\varphi}&= \chi_H
        \end{aligned}
        \right.
    \end{equation*}
     satisfies
    \begin{equation}
        C_1\left( \log \frac{1}{\o(E)} \right)^{\frac12}\leq A_\gamma^{\gamma \ell(\Delta)}(u) (x,\varphi(x))\hbox{   for all    }(x,\varphi(x))\in E. 
        \label{eqn.1.9}
    \end{equation}
     Here $\ell(\Delta)=\ell(Q)$ by definition and 
    \begin{equation*}
        A_\gamma^{\gamma \ell(\Delta)}(u)(x,\varphi(x))=\left( \int_{\wt \Gamma_\gamma^{\gamma \ell(\Delta)}(x,\varphi( x))}\lvert \nabla u(Z)\rvert^2\lvert Z-(x,\varphi(x))\rvert^{2-n}dZ\right)^{\frac12}.
    \end{equation*}
    
    \end{cor}

    Corollary \ref{cor.1.8} follows immediately from Proposition \ref{prop.1.1} by the change of variables $\Phi$. Note that in particular $\om^{A_Q}_{\R_n^+}(\Phi(E))=\o(E)$. To obtain our main results, which are consequences of Proposition \ref{prop.1.1} and Corollary \ref{cor.1.8}, we return for now to the case of the upper half-plane $\R_+^n$.

    \begin{thm}\label{thm.1.10}
        Fix a cube $Q\ss \R^{n-1}$. Assume that for all solutions $u$ to
    \begin{equation*}
        \tag{DP}
        \label{eqn.dp''}
        \left\{
        \begin{aligned}
            Lu&= 0 \text{ in }\R_+^n\\
            u\vert_{\R_+^{n-1}}&= \chi_H
        \end{aligned}
        \right.
    \end{equation*}
    where $H\ss Q$ is a Borel set, the following estimate holds
    \begin{equation}
        \frac{1}{\lvert (1+\gamma)Q\rvert }\int_{T\left( (1+\gamma)Q \right)}t\lvert \nabla u(y,t)\rvert^2dydt\leq A,
        \label{eqn.1.11}
    \end{equation}
    where $A$ is fixed and $\gamma$ is the constant depending on ellipticity and dimension from Proposition \ref{prop.1.1}. Then if $E\ss Q$ is a Borel set, $\o=\o^{A_Q}$, and $\o(E)<\delta_0$ with
    $\delta_0$ is as in Proposition \ref{prop.1.1}, we have
    \begin{equation}
        \frac{\lvert E\rvert}{\lvert Q\rvert}\leq C\left[ \log \frac{1}{\o(E)} \right]^{-1},
        \label{eqn.1.12}
    \end{equation}
    where $C$ depends only on dimension, ellipticity and $A$.
    \end{thm}

    \begin{proof}
        Let $E$ be as above, $H$ as in Proposition \ref{prop.1.1}. Then, provided $\o(E)<\delta_0$ \eqref{eqn.1.2} ensures that 
        \begin{equation*}
            C_1\left[ \log \frac{1}{\o(E)} \right]\leq S_\gamma^{\gamma \ell(Q)}(u)^2 (x,0),
        \end{equation*}
        for all $(x,0)\in E$. Integrating over $E$, by Fubini, we obtain
        \begin{equation*}
            C_1\left[ \log \frac{1}{\o(E)} \right]\lvert E\rvert \leq \iint_{T\left( (1+\gamma)Q \right)}t\lvert \nabla u(y,t)\rvert^2dydt\leq A\lvert \left( 1+\gamma \right)Q\rvert,
        \end{equation*}
        by \eqref{eqn.1.11}. Theorem \ref{thm.1.10} follows.
    \end{proof}

    \begin{cor}\label{cor.1.13}
       With the notation of Theorem \ref{thm.1.10}, assume that \eqref{eqn.1.11} holds for any cube $Q$, that is the solution of \eqref{eqn.dp''} with data characteristic function of a bounded Borel set, 
       verifies the Carleson measure condition. Then for any cube $Q_0\subset\R^{n-1}$, $\o=\o^{A_{Q_0}} \in A_{\infty}(Q_0)$. That is for every $\varepsilon\in (0,1)$ there exists $\delta\in(0,1)$ such that for any 
      cube $Q\subset Q_0$ and $E\subset Q$ with 
     \begin{equation} \label{eqn.1.14}
\frac{\o(E)}{\o(Q)}<\delta   \hbox{    then   }  \frac{|E|}{|Q|}<\varepsilon.
          \end{equation}
          Moreover, there exists $p_0>1$ depending on the ellipticity, the dimension and $A$, such that, for $p_0\leq p<\infty$, the solution of \eqref{eqn.dp''} with data $F$ continuous and of compact support satisfies
        \begin{equation}
            \lVert u^*\rVert_{L^p(\R^{n-1})}\leq \wt C\lVert F\rVert_{L^p(\R^{n-1})},
            \label{eqn.1.15}
        \end{equation}
        with $u^*(x,0)=\sup_{X\in \Gamma_\gamma(x,0)}\lvert u(X)\rvert$ and $\wt C$ depending on ellipticity, dimension, $A$ and $p$.
    \end{cor}

    \begin{proof}[Proof of Corollary \ref{cor.1.13}]
    Recall (see \cite{KenigCBMS}) that $\o\in A_\infty(\o^{A_Q})$ thus given $\eta>0$ there is $\delta>0$ so that if $\frac{\o(E)}{\o(Q)}<\delta$ then $\frac{\o^{A_Q}(E)}{\o^{A_Q}(Q)}<\eta$. We consider $\eta<\delta_0$ to be specified
    and $\delta_0$ as in Theorem \ref{thm.1.10}.
    In this case $\o^{A_Q}(E)<\eta \o^{A_Q}(Q)\le \eta<\delta_0$ and \eqref{eqn.1.12} implies that choosing $\eta$ small enough depending on $\varepsilon$ we have 
     \begin{equation}\label{eqn1.15-tt}
  \frac{|E|}{|Q|}\le C\left[ \log \frac{1}{\o^{A_Q}(E)} \right]^{-1}\le C\left(\log\frac{1}{\eta}\right)<\varepsilon,
          \end{equation}
          which proves \eqref{eqn.1.14}.
     \eqref{eqn.1.15} in turn follows from \eqref{eqn.1.14} by well known arguments (see \cite{KenigCBMS}, \cite{HKMP}, \cite{Kenig-Shen}, the discussion around (4.3).)
    \end{proof}
    
    Recall that the fact that $\o^{A_{Q}} \in A_{\infty}(Q)$ as in \eqref{eqn.1.14} ensures that estimate \eqref{eqn.1.17} holds (see \cite{KenigCBMS}). Theorem \ref{thm.1.16} ensures that the converse is true and provides  a necessary and sufficient condition for the elliptic measure of a second order differential operator in the upper half plane to be an $A_\infty$ weight with respect to the surface measure to the boundary.

    \begin{thm}\label{thm.1.16}
        Assume that either for some $p$, $1+\frac{1}{n-2}\leq p<\infty$ when $n\geq 3$ or  $p_0\leq p<\infty$ with $p_0$ depending only on ellipticity, when $n=2$, 
        the following estimate holds
        \begin{equation}
            \lVert S_\gamma(u)\rVert_{L^p\left( \R^{n-1} \right)}\leq A\lVert u^*\rVert_{L^p\left( \R^{n-1} \right)},
            \label{eqn.1.17}
        \end{equation}
        where $\gamma$ as in Proposition \ref{prop.1.1} and
        \begin{equation*}
        S_\gamma(u)(x,0)=\left( \int_{\Gamma_\gamma(x,0)}\lvert \nabla u(Z)\rvert^2\lvert Z-(x,0)\rvert^{2-n}dZ \right)^{\frac12},
        \end{equation*}
         for all solutions $u$ to \eqref{eqn.dp''} with data $\chi_H$, where $H$ is a bounded Borel set. Then there exist positive constants $\delta_1$ and $C_1$
         depending only on ellipticity, dimension, $A$ and $p$ such that for any cube $Q\subset \R^{n-1}$ and any Borel set $E\subset Q$ with $\o(E)=\o^{A_Q}(E)<\delta_1$
         \begin{equation}\label{eqn-tt-1.16}
          \frac{|E|}{ |Q|} \le C_1\left(\log\frac{1}{\o(E)}\right)^{-\frac{p}{2}}
         \end{equation}
         Hence for every cube $Q\subset \R^{n-1}$, $\o^{A_Q}\in A_\infty(Q)$.
         
    \end{thm}

    \begin{proof}
        Fix a cube $Q\ss \R^{n-1}, \gamma, \delta_0$ as in Propostion \ref{prop.1.1}. Let $E\ss Q$ be a Borel set, $\o=\o^{A_Q}$ and assume $\o(E)\leq \delta<\delta_0$. Let $H$ be the Borel set in Proposition \ref{prop.1.1}. Then, by \eqref{eqn.1.2}, for all $(x,0)\in E$ we
        \begin{equation}
            C_1\left[ \log \frac{1}{\o(E)} \right]\leq S_\gamma^{\gamma \ell(Q)}(u)^2(x,0).
            \label{eqn.1.18}
        \end{equation}
        Now, by Lemma 4.9 in \cite{HKMP}, we have for $X\in \R_+^n$ with $\lvert X-(x_Q,0)\rvert \geq 3\ell(Q)$
        \begin{equation}
            \lvert u(x)\rvert \leq C\left( \frac{\ell(Q)}{\lvert X-(x_Q,0)\rvert} \right)^{n-2+\alpha},
            \label{eqn.1.19}
        \end{equation}
         where $\alpha>0$ depends only on the dimension and the ellipticity. Consider now, for $(x,0)\in E$,
        \begin{align}\label{eqn1.18A}
            & \int_{\Gamma_\gamma(x,0)\bs\Gamma_\gamma^{\gamma \ell(Q)}(x,0)}\lvert \nabla u\rvert^2 (Y)\frac{1}{\lvert Y-(x,0)\rvert^{n-2}}dY\\
            &= \sum_{j=0}^\infty \int_{\left\{ Y=(y,t): \lvert x-y\rvert < \gamma t,\, 2^j\ell(Q)\leq t\leq 2^{j+1}\ell(Q)  \right\}}\lvert \nabla u(Y)\rvert^2 \frac{1}{\lvert Y-(x,0)\rvert^{n-2}}dY\nonumber\\
            &= \sum_{j=0}^\infty R_j.\nonumber
        \end{align}
        We proceed to estimate $R_j$, for each $j$. If $Y$ belongs to the region of the integration defining $R_j$, we have $\lvert Y-(x,0)\rvert^{n-2}\simeq \left[ 2^j\ell(Q) \right]^{n-2}$. Using the Cacciopoli estimate ((1.3) in \cite{KKPT}), combined with \eqref{eqn.1.19} we see that
        \begin{equation}\label{eqn1.19A}
            R_j\lesssim \left( 2^j\ell(Q) \right)^{-(n-2)}\cdot \left( 2^j\ell(Q) \right)^{-2}\cdot \left( 2^j\ell(Q) \right)^n\cdot \left[ \frac{\ell(Q)}{2^j \ell(Q)} \right]^{n-2+\alpha}=2^{-j (n-2+\alpha)}.
        \end{equation}
        Thus, 
        \begin{equation}\label{eqn1.19B}
            \int_{\Gamma_\gamma(x,0)\bs \Gamma_\gamma^{\gamma \ell(Q)}(x,0)}\lvert \nabla u(Y)\rvert^2 \frac{1}{\lvert Y-(x,0)\rvert^{n-2}} dy\leq C,
        \end{equation}
        where $C$ depends on ellipticity and dimension. Combining \eqref{eqn1.18A}, \eqref{eqn1.19B} with \eqref{eqn.1.18}, and taking $\delta$ small enough, depending only on ellipticity and dimension, we obtain
        \begin{equation}
            \wt C_1\left[ \log \frac{1}{\o(E)} \right]\leq S_\gamma(u)^2 (x,0),
            \label{eqn.1.20}
        \end{equation}
        for all $(x,0)\in E$. We now take the $\frac{p}{2}$ power of both sides of $\eqref{eqn.1.20}$ and integrate over $E$ with respect to Lesbesgue measure in $\R^{n-1}$. We obtain
        \begin{equation}
            \wt C_1^{\frac p2}\left[ \log \frac{1}{\o(E)} \right]^{\frac12}\lvert E\rvert \leq \int_E S_\gamma(u)^p \, dx
            \label{eqn.1.21}
        \end{equation}
        We now use \eqref{eqn.1.17} to bound the right hand side of \eqref{eqn.1.21} by
        \begin{equation}\label{eqn1.21A}
            A^p\int_{\R^{n-1}}\left( u^* \right)^p \, dx=A^p\int_{3Q}\left( u^* \right)^p  \, dx+A^p\int_{\R^{n-1}\bs 3Q}\left( u^* \right)^p \, dx = \text{I} + \text{II}.
        \end{equation}
        Since $0\leq u\leq 1$
        \begin{equation}\label{eqn1.21'}
         \text{I}\leq C_nA^p\lvert Q\rvert.
         \end{equation} 
         For II, we use the estimate \eqref{eqn.1.19}. If $x\in \R^{n-1}\bs 3Q$, $Y\in \Gamma_\gamma(x,0)$, we know that $\lvert Y-(x,0)\rvert \simeq \text{dist} (Y,\R^{n-1}) \leq \lvert Y-(x_Q,0)\rvert$. Then
         \begin{equation}\label{eqn1.21B}
          3\ell(Q)\le \lvert (x,0)-(x_Q,0)\rvert \leq \lvert (x,0)-Y\rvert +\lvert Y-(x_Q,0)\rvert \lesssim \lvert Y-(x_Q,0)\rvert
         \end{equation}
          so that 
          \begin{equation}\label{eqn1.21C}
          u^*(x,0)\lesssim \left[ \frac{\ell(Q)}{\lvert x-x_Q\rvert} \right]^{n-2+\alpha}
          \end{equation} Hence,
        \begin{eqnarray}\label{eqn1.21D}
            \text{II}& \lesssim& A^p \int_{\lvert x-x_Q\rvert\geq 3\ell(Q)} \left[ \frac{\ell(Q)}{\lvert x-x_Q\rvert} \right]^{p\left[ n-2+\alpha \right]}\\
            &            \lesssim &
           A^p \sum_{j=1}^\infty \int_{3^j\ell(Q)\le\lvert x-x_Q\rvert\le 3^{j+1}\ell(Q)} \left[ \frac{\ell(Q)}{3^j\ell(Q)} \right]^{p\left[ n-2+\alpha \right]}\, dx\nonumber\\
           &\lesssim& A^p |Q|\sum_{j=1}^\infty 3^{-j(p(n-2+\alpha) -n +1}
           \lesssim
            A^p\lvert Q\rvert,\nonumber
        \end{eqnarray}
        since $p\geq 1+(n-2)^{-1}(n\geq 3)$ or $p>\frac{1}{\alpha}$ when  $n=2$. Combining \eqref{eqn.1.21}, \eqref{eqn1.21A}, \eqref{eqn1.21'} and \eqref{eqn1.21D} we obtain
        \begin{equation}
            \wt C_1^{\frac p2}\left[ \log \frac{1}{\o(E)} \right]^{\frac p2}\lvert E\rvert \leq C\lvert Q\rvert,
            \label{eqn.1.22}
        \end{equation}
        with $C$ and $\wt C_1$ depending on ellipticity, dimension, $p$, $A$. The conclusion of Theorem \ref{thm.1.16} as in \eqref{eqn1.15-tt}.
            \end{proof}
            
    \begin{rmk}\label{lip-rmk}
        Theorem \ref{thm.1.10}, Corollary \ref{cor.1.13} and Theorem \ref{thm.1.16} have corresponding version in domains $D_\varphi$ with $\varphi$ Lipschitz (see \eqref{lip}). This follows from the change of variable argument use in Corollary \ref{cor.1.8}.        
    \end{rmk}

    \section{Bounded Lipschitz Domains}
    \setcounter{equation}{0}
    In this section we establish variants of the results in Section \ref{sec.1}, valid for bounded Lipschitz domains $\Omega\ss \R^n$. A bounded domain $\Omega\subset\R^n$ is said to be Lipschitz if there exists $R>0$ such that for all $P\in \partial \Omega$ there is an $(n-1)$-plane $L(P)$ through $P$ and a Lipschitz function $\varphi_P$ defined on $L(P)$ such that 
    \begin{equation}\label{bdd-lip-dom1}
    \Omega\cap B(P,2R)=\left\{ (x,t)\in B(P,2R): x\in L(P),\ t>\varphi_P(x) \right\}.
    \end{equation}
  Since $\Omega$ is bounded, $\partial \Omega$ can be covered by finitely many balls $\{B(P_i, R_i)\}_i$ with $P_i\in\partial\Omega$ and $|P_i-P_j|\ge R$. Let $M=\max_{1\le i\le m} \rm{Lip}\varphi_{P_i}$. Since for every $P\in\partial\Omega$ there exists $i=1,\cdots, m$ such that $|P-P_i|<R/2$ then $B(P,R)\subset B(P_i, 2R)$ for some $i$ and 
  \begin{equation}\label{bdd-lip-dom2}
    \Omega\cap B(P,R)=\left\{ (x,t)\in B(P,R): x\in L(P_i),\ t>\varphi_{P_i}(x) \right\}.
    \end{equation}
    Hence \eqref{bdd-lip-dom2} ensures that  there exists $M>0$ as above such that for each $P\in\partial\Omega$, $\rm{Lip}\varphi_P\le M$ where $\varphi_P$ is a Lipschitz function used to represent $\Omega\cap B(P,R/2)$ as in \ref{bdd-lip-dom1}. We refer to $R$, $m$ and $M$ as the Lipschitz character of the domain $\Omega$.

\bigskip
    We start with the analog of Proposition \ref{prop.1.1}. We assume without loss of generality that $0\in \Omega$ and let $\o=\o^0$, be the elliptic measure with pole at $0$.

    \begin{prop}\label{prop.2.1}
        Let $\Omega\ss \R^n$ be a bounded Lipschitz domain, and let $L$ an elliptic operator. Assume that $\Delta$, a surface ball on $\partial \Omega$ has radius less than $R$ (as above).
    There are constants $\wt \delta_0>0$, $\gamma>0$, depending only on the ellipticity, the dimension and the Lipschitz character of $\Omega$ such that 
         if $E\ss \Delta$ is a Borel set with $\omega(E)\le \delta\omega(\Delta)$ where $\delta<\wt\delta_0$, there exists 
       a Borel set $H\ss \Delta$ such that the solution $u$ of
     \begin{equation*}
        \tag{DP}
        \label{eqn.dp'''}
        \left\{
        \begin{aligned}
            Lu&= 0 \text{ in }\Omega\\
            u\vert_{\partial \Omega}&= \chi_H
        \end{aligned}
        \right.
    \end{equation*}
satisfies
    \begin{equation}
        \wt C_1 \ol k^{\frac12}\leq A_\gamma^{\gamma \ell(\Delta)}(u)(P), \text{ for all } P\in E.
        \label{eqn.2.2}
    \end{equation}
    Here 
    \begin{equation*}
        A_\gamma^{\gamma \ell(\Delta)}(u)(P)=\left( \int_{\ol \Gamma_\gamma^{\gamma l(\Delta)}(P)}\lvert \nabla u(Y)\rvert^2\lvert Y-P\rvert^{2-n}dy \right)^{\frac12},
    \end{equation*}
     $\ol k= \log  \frac{\o(\Delta)}{\o(E)} $ and $\wt C_1$ depends only on the ellipticity, the dimension and the Lipschitz character of $\Omega$.
   \end{prop}

   \begin{proof}
       Let $\Delta=B(P_{\Delta},r)\cap\partial\Omega$ with $P_\Delta\in\partial\Omega$, $r\le R$. Let $\varphi_{P_\Delta}=\varphi$. Without loss of generality we may assume that $L(P_\Delta)=\R^{n-1}$. Then 
       \[
       D_\varphi=\{(x,t): x\in\R^{n-1},\ t>\varphi(x)\}\]
       with $\|\nabla \varphi\|_\infty\le M$ (see \eqref{lip}).
       Let $v$ be the solution to
     \begin{equation*}
         \tag{$DP_{D_\varphi}$}
         \label{eqn.dp4}
        \left\{
        \begin{aligned}
            Lv&= 0 \text{ in }D_\varphi\\
            v\vert_{\partial D_\varphi}&= \chi_H
        \end{aligned}
        \right.
    \end{equation*}
    Consider $w=v-u$, in the domain 
    \begin{equation*}
        T(\gamma \Delta)=\left\{ X=(x,t): \lvert x-P_\Delta\rvert<\gamma \ell(\Delta) , \varphi(x)\leq t\leq \varphi(x)+\gamma \ell(\Delta)\right\},
    \end{equation*}
    where $P_\Delta$ is the center of $\Delta$, $\ell(\Delta)=r$ and $\gamma\in(0,1)$ is as in Corollary \ref{cor.1.8}. 
    Then, $Lw=0$ in $T\left( 2\gamma \Delta \right)$, and $\lvert w\rvert \leq 2$, $w\vert_{2\gamma\Delta}\equiv 0$. We now apply the boundary H\"older continuity estimate for $w$ in $T\left( 2\gamma\Delta \right)$ (Estimate (1.9) in \cite{KKPT}, stated for non-negative solutions, but valid for the variable sign solutions in the form stated below) for $X\in T(\gamma\Delta)$
    \begin{equation}
        \lvert w(X)\rvert \leq C\left[ \frac{\text{dist}(X,\partial T(2\gamma\Delta))}{\ell(\Delta))} \right]^\beta\sup_{T(2\gamma\Delta)}\lvert w\rvert,
        \label{eqn.2.3}
    \end{equation}
    where $C,\beta$ depend only on the ellipticity, the dimension and the Lipschitz character on $\Omega$. Notice also that if $X\in T\left( \gamma\Delta \right)$, then $\text{dist}(X, \partial T\left( 2\gamma\Delta) \right)\simeq \text{ dist }(X, \partial D_\varphi)\simeq \text{dist}(X, 2\gamma\Delta)$, with comparability constants depending only on the Lipschitz character of $\Omega$.

    Consider now $\int_{\wt \Gamma_\gamma^{\gamma \ell(\Delta)}(P)}\lvert \nabla w(Y)\rvert^2 \lvert Y-P\rvert^{2-n}dY$, for $P\in E$ . We split $\wt \Gamma_\gamma^{\gamma \ell(\Delta)}(P)$ into the disjoint subregions $R_j$ with $j\ge 0$ such that 
    \begin{equation*}
    R_j=\left\{Y\in \wt \Gamma_\gamma^{\gamma \ell(\Delta)}(P): Y\in B(P, 2^{-j}\gamma \ell(\Delta))\backslash B(Y, 2^{-j-1}\gamma \ell(\Delta))\right\} 
        \end{equation*}
       Note that for $Y\in R_j$, $\lvert Y-P\rvert^{2-n}\simeq \left( 2^{-j}\gamma \ell(\Delta) \right)^{2-n}$.

     Combining Caccioppoli's estimate on each $R_j$ with the fact that $|w|\le 2$ and \eqref{eqn.2.3} we obtain
    \begin{eqnarray}
        \int_{\wt \Gamma_\gamma^{\gamma \ell(\Delta)}(P)}\lvert \nabla w(Y)\rvert^2 \lvert Y-P\rvert^{2-n}dY &\lesssim & \sum_{j=1}^\infty  \frac{1}{( 2^{-j}\gamma \ell(\Delta))^{n-2}}\int_{\wt \Gamma_\gamma^{\gamma \ell(\Delta)}(P)\cap R_j}\lvert \nabla w(Y)\rvert^2 dY \nonumber\\
&\lesssim & \sum_{j=1}^\infty  \frac{1}{( 2^{-j}\gamma \ell(\Delta))^{n}}\int_{\wt \Gamma_\gamma^{\gamma \ell(\Delta)}(P)\cap (R_{j-1}\cup R_j\cup R_{j+1})}\lvert w(Y)\rvert^2 dY \nonumber\\
&\lesssim & \sum_{j=1}^\infty  \frac{1}{( 2^{-j}\gamma \ell(\Delta))^{n}} (2^{-j}\gamma)^\beta {( 2^{-j+1}\gamma \ell(\Delta))^{n}} \le C.
        \label{eqn.2.4}
    \end{eqnarray}
    where $C$ depends only on the ellipticity, the dimension and the Lipschitz character of $\Omega$. We denote by $A_\Delta$ is the non-tangential point corresponding to $\Delta$, and by
    $\o^X$ and $\o_{D_\varphi}^X$ the elliptic measures for the domains $\Omega$ and $D_\varphi$ respectively.

     Using (1.13), (1.14) and Theorem 1.11 in \cite{KKPT} we obtain
    \begin{equation}
        \frac{\o^{A_\Delta}(E)}{\o_{D_\varphi}^{A_\Delta}(E)}\simeq 1,
        \label{eqn.2.5}
    \end{equation}
    with comparability constants that depend only on the ellipticity, the dimension and the Lipschitz character of $\Omega$. Moreover, (1.15) in \cite{KKPT} ensures that 
  \begin{equation}   
    \frac{\o(E)} {\o(\Delta)} \simeq \o^{A_\Delta}(E)
 \label{eqn.2.5A}
    \end{equation}   
     with comparability constants depending only on the ellipticity and the dimension. Combining \eqref{eqn.2.5}, \eqref{eqn.2.5A} and  Corollary \ref{cor.1.8}, we have that, if $\frac{\o(E)}{\o(\Delta)}<\wt \delta_0$, then $\o^{A_\Delta}_{D_\varphi }(E)\lesssim\wt \delta_0$. Hence 
     by \eqref{eqn.1.9} and \eqref{eqn.2.5A}
     $C_1\left[ \log \frac{\o(E)}{\o(\Delta)} \right]^{\frac12}\leq A_\gamma^{\gamma \ell(\Delta)}(v)(P)$, for all $P\in E$. Using \eqref{eqn.2.4} since 
     $A_\gamma^{\gamma \ell(\Delta)}(u)(P)\le A_\gamma^{\gamma \ell(\Delta)}(v)(P) +A_\gamma^{\gamma \ell(\Delta)}(w)(P)$,
     taking $\wt \delta_0$ possibly smaller, still depending only on the ellipticity, the dimension and the Lipschitz character of $\Omega$,
    \begin{equation*}
        C_1\left[ \log \frac{\o(\Delta)}{\o(E)} \right]^{\frac12}\leq A_\gamma^{\gamma \ell(\Delta)}(u)(P) \qquad \hbox{ for all  } P\in E,
    \end{equation*}
    as desired.
   \end{proof}

   \begin{thm}\label{thm.2.5}
      Let $\Omega$ be a bounded Lipschitz domain, $L$ an elliptic operator. Assume $0\in \Omega, \o=\o^0$ is the elliptic measure in $\Omega$ corresponding to $L$. Assume that for all Borel sets $H\ss \partial \Omega$, the solution to the Dirichlet problem
      \begin{equation*}
        \left\{
        \begin{aligned}
            Lu&= 0 \text{ in }\Omega\\
            u\vert_{\partial \Omega}&= \chi_H
        \end{aligned}
        \right.
    \end{equation*}
satisfies the following Carleson bound
    \begin{equation}
        \sup_{\Delta \ss \partial \Omega, \
        \rm{ diam } (\Delta)\leq \rm {diam }(\Omega)} \frac{1}{\sigma(\Delta)} \int_{T(\Delta)}\text{dist}(X,\partial \Omega)\lvert \nabla u(X)\rvert^2 dX\leq A.
        \label{eqn.2.6}
    \end{equation}
    Then $\o\in A_\infty(d\sigma)$, i.e., there exist $0<\alpha<1$, $0<\beta<1$ and $r_0>0$ such that for all surface balls $\Delta$ of diameter smaller than $r_0$, we have, for all Borel sets $E\ss \Delta$ that $\frac{\o(E)}{\o(\Delta)}<\alpha\implies \frac{\sigma(E)}{\sigma(\Delta)}<\beta$ with $\alpha, \beta$ depending only on ellipticity, dimension, Lipschitz character and $A$. Moreover, there exists $p_0>1$ such that the solution of the Dirichlet problem
       \begin{equation*}
        \left\{
        \begin{aligned}
            Lv&= 0 \text{ in }\Omega\\
            v\vert_{\partial \Omega}&= f\in C(\partial \Omega) 
        \end{aligned}
        \right.
    \end{equation*}
    verifies the estimate
    \begin{equation*}
        \int_{\partial \Omega}(v^*)^pd\sigma\leq C\int_{\partial \Omega}f^pd\sigma,
    \end{equation*}
    for $p_0\leq p<\infty$, with $C$ depending only on the ellipticity, the dimension, the Lipschitz character of $\Omega$, $A$ and $p$.
   \end{thm}

   \begin{proof}
       Choose $\alpha<\wt \delta_0$, where $\wt\delta_0$ is as in Proposition \ref{prop.2.1}. Choose $r_0$ the radius of $\Delta$ small enough so small that $2\gamma r_0< R$, with $R$ as in \eqref{bdd-lip-dom1}. In this case $2\gamma\Delta$ can be regarded as a surface ball for some $D_\varphi$. Apply now Proposition \ref{prop.2.1} and integrate the square of \eqref{eqn.2.2} over $E$. By Fubini, we obtain
       \begin{eqnarray*}
           \wt C_1\left[\log  \frac{\o(\Delta)} {\o(E)}\right] \sigma(E)&\lesssim& \int_{T\left( (1+\gamma) \Delta\right)}\text{ dist }(X,\partial \Omega)\lvert \nabla u(X)\rvert^2 dX\\
           &\leq & A\sigma\left( (1+\gamma)\Delta \right)\lesssim A\sigma(\Delta),
       \end{eqnarray*}
       by assumption \eqref{eqn.2.6}. The conclusion holds choosing $\alpha$ so small that 
       $$\frac{C}{\wt C_1}\left[ \log  \frac{\o(\Delta)} {\o(E)} \right]^{-1}<1.
       $$
        The rest of the result follows from the theory of $A_\infty$ weights (\cite{CF}) and well known results (see \cite{KenigCBMS}).

   \end{proof}

   In a similar way we can prove:
   \begin{thm}\label{thm.2.7}
    Let $\Omega$ be a bounded Lipschitz domain, $L$ an elliptic operator. Assume $0\in \Omega, \o=\o^0$ is the elliptic measure in $\Omega$ corresponding to $L$. Assume that for all Borel sets $H\ss \partial \Omega$, the solution to the Dirichlet problem
      \begin{equation*}
        \left\{
        \begin{aligned}
            Lu&= 0 \text{ in }\Omega\\
            u\vert_{\partial \Omega}&= \chi_H
        \end{aligned}
        \right.
    \end{equation*}
  Suppose that 
         \begin{equation}
           \lVert A_\gamma(u)\rVert_{L^q(\partial \Omega)}\leq A\lVert u^*\rVert_{L^q(\partial \Omega)},
           \label{eqn.2.8}
       \end{equation}
       for some $p$, $1+\frac{1}{n-2}\leq q<\infty$ if $n \geq 3$ or $q_0\leq p<\infty$ if $n=2$, with $q_0$ depending on the ellipticity and the Lipschitz character of $\Omega$. 
        Here $\gamma$ is taken as in Proposition \ref{prop.2.1}.
   Then $\o\in A_\infty(d\sigma)$.
    Moreover there exists $p_0$ such that the solution of the Dirichlet problem
       \begin{equation*}
        \left\{
        \begin{aligned}
            Lv&= 0 \text{ in }\Omega\\
            v\vert_{\partial \Omega}&= f\in C(\partial \Omega) 
        \end{aligned}
        \right.
    \end{equation*}
satisfies
    \begin{equation*}
        \int_{\partial \Omega}(v^*)^pd\sigma\leq C\int_{\partial \Omega}f^p\, d\sigma,
    \end{equation*}
    for $p_0\le p<\infty$ with $C$ depending only on the ellipticity, the dimension, the Lipschitz character of $\Omega$, $A$ and $p$.
   \end{thm}

   The proof of \ref{thm.2.7} is the same as the one of Theorem \ref{thm.1.16}, using Proposition \ref{prop.2.1} instead of Proposition \ref{prop.1.1}, and the following analog of Lemma 4.9 in \cite{HKMP}.
For the notation used in the following lemma we refer the reader to the beginning of this section.  

   \begin{lma}\label{lma.2.9}
       Let  $u$ be the solution to $L u=0$ in $\Omega$ a bounded Lipschitz domain, with boundary values $\chi_H$, $H$ a Borel set, contained in a surface ball $\Delta=B(P_\Delta, r)\cap\partial \Omega$ with $r<R/4$.  
 Then, there exist positive constants $C$ and $\beta$, depending only on dimension and ellipticity such that
      \begin{equation}
      u(X)\leq C\left[  \frac{\ell(\Delta)}{\lvert X-P_\Delta\rvert}\right]^{n-2+\beta} \hbox{ for }X\in T\left( \frac{1}{2}\Delta_0 \right)\bs T(2\Delta)
      \end{equation}
      and 
      \begin{equation}
      u(X)\leq C\left[ \frac{\ell(\Delta)}{R} \right]^{n-2+\beta}  \hbox{ for }X\in \Omega\bs T\left( \frac{1}{2}\Delta_0 \right).
      \end{equation}

   \end{lma}
   
   Note that by the maximum principle in this case
   $0\leq u\leq 1$. Moreover given the definition of $R$ (see \eqref{bdd-lip-dom1}) in this case $\Delta_0=B(P_\Delta, R)\cap \Omega$ can be seen as the area above a Lipschitz graph inside $B(P_\Delta, R)$. 
   The proof of Lemma \ref{lma.2.9} is identical to the one of Lemma 4.9 in \cite{HKMP} and it is therefore omitted.

\end{document}